\documentclass[11pt]{amsart}
\usepackage{mathabx}
\usepackage[usenames,dvipsnames,svgnames,table]{xcolor} 
\usepackage{amssymb,amsfonts}
\usepackage{amsmath,amsthm}

\usepackage{pgf, tikz} 
\usepackage[active]{srcltx}		
\usepackage{pdfsync}					
\usepackage{enumitem}
\usepackage[colorinlistoftodos]{todonotes}
\usepackage{dsfont}
\usepackage{amsrefs}

\usepackage[pdfdisplaydoctitle,colorlinks,urlcolor=blue,linkcolor=blue,citecolor=blue]{hyperref}

\newcommand{\ddd}{\,{\mathbf d}}
\newcommand{\dd}{\,{\rm  d}} 
\newcommand{\Tr}{\,{\rm Tr}}
\newcommand{\Lip}{\,{\rm Lip}}

\newcommand{\s}{\sigma}
\renewcommand{\d}{\delta}

\newcommand{\la}{\lambda}
\renewcommand{\a}{\alpha}
\renewcommand{\b}{\beta}
\renewcommand{\k}{\kappa}

\newcommand{\p}{\partial}
\newcommand{\D}{\Delta}

\newcommand\R{{\mathbb{R}}}
\renewcommand\P{{\mathcal{P}}}
\newcommand\N{{\mathbb{N}}}

\newcommand\Cc{{\mathcal{C}}}

\newcommand\E{{\mathbb{E}}}
\newcommand\Z{{\mathbb{Z}}}
\newcommand\T{{\mathbb{T}}}
\newcommand\F{{\mathcal{F}}}

\renewcommand\div{{\rm div}}
\newcommand\weak{{\rm weak}}

\newcommand\X{{\mathcal{X}}}
\newcommand\B{{\mathcal{B}}}
\newcommand\A{{\mathcal{A}}}

\newtheorem{theorem}{Theorem}[section]
\newtheorem{proposition}[theorem]{Proposition}
\newtheorem{lemma}[theorem]{Lemma}
\newtheorem{corollary}[theorem]{Corollary}

\theoremstyle{definition}
\newtheorem{definition}[theorem]{Definition}

\theoremstyle{remark}
\newtheorem{remark}[theorem]{Remark}
\numberwithin{equation}{section}

\begin{document}

\title[On quasi-stationary Mean Field Games models]
{On quasi-stationary Mean Field Games models}

\author{Charafeddine Mouzouni}
\address{Univ Lyon, \'Ecole centrale de Lyon, CNRS UMR 5208, Institut Camille Jordan, 36 avenue Guy de Collonge, F-69134 Ecully Cedex, France.}
\email{\href{mailto:mouzouni@math.univ-lyon1.fr}{mouzouni{@}math.univ-lyon1.fr}}


\date{\today}   

\begin{abstract}
We explore a mechanism of decision-making in Mean Field Games with myopic players.
At each instant, agents set a strategy  
which optimizes their expected future cost by assuming their environment as immutable.  As the system evolves, the players observe the evolution of the system and adapt to their new environment without anticipating.
With a specific cost structures, these models give rise to coupled systems of partial differential equations of quasi-stationary nature. We provide sufficient conditions for the existence and uniqueness of classical solutions for these systems, and give a rigorous derivation of these systems from $N$-players stochastic differential games models. Finally, we show that  the population can self-organize and converge exponentially fast to the ergodic Mean Field Games equilibrium, if the initial distribution is sufficiently close to it and the Hamiltonian is quadratic.
%
%
\end{abstract}


\maketitle


\section{Introduction }\label{section1}

The Mean Field Games formalism has been introduced some years ago by series of seminal papers by J.-M. Lasry and P.-L. Lions \cite{LLMFG,LLMFG1,LLMFG2} and M. Huang, R. Malham{\'e }and P. Caines  \cite{Malhame,Malhame2}. It describes the evolution of stochastic differential games with a large population of rational players and where the strategies of the agents are not only affected by their own preferences but also by the state of the other players through a global mean field effect. In terms of partial differential equations, these models are typically a system of  a  transport or Fokker-Plank equation for the distribution of the agents coupled with a Hamilton-Jacobi-Bellman equation.

The motivation of this paper is to study a strategy-choice mechanism that is different from classical Mean Field Games. Our agents are myopic, and choose their actions according to the information available at time $t$, by fixing the future state of their opponents and trying to get the best possible gain in the future $(s> t)$. 
Players anticipate no evolution of the system, undergo changes in their environment and adapt their strategies. A system of interacting agents can have such irrational behavior in panic situations for instance.
In this framework, agents build a strategy at each moment and the global (in time) strategy is the history of all the chosen strategies. 
This decision-making mechanism intrinsically implies the existence of two time scales: a fast time scale which is linked to the optimization of the  expected future cost; and a slow time scale linked to the actual evolution of the system. The coexistence of these two time scales gives rise to equations of quasi-stationary nature.  
We are also interested in the formation of equilibria for this type of evolution systems, and in the rate at which these systems converge towards these equilibria.

In general  the decision-making mechanism in mean field games (MFG for short) involves solving a stochastic control problem, that provides a global in time optimal strategy. In the case where the players aim to minimize a long time average cost, it is well known that the MFG system of partial differential equations is stationary and takes the following form \cite{LLMFG,LLMFG1, Feleqi,Feleqi2},
\begin{equation}
\label{EMFG-intro} 
\left\{
\begin{aligned}
&-\s \D \bar{u}+H(x,D\bar{u})+\bar{\la}=F(x,\bar{m})\quad \mbox{ in }Q:=\T^d \\
\\
&-\s \D \bar{m}-\div(\bar{m}H_{p}(x,D\bar{u}))=0 \quad \mbox{ in }Q\\      
\\
&\bar{m}\geq 0,\quad <\bar{m}>:=\int_{Q}\bar{m} =1,\quad <\bar{u}>=0 \qquad 
\end{aligned}
\right.
\end{equation} 
Here $\s>0$, all the functions are $\Z^{d}$-periodic, the unknowns are the constant $\bar{\la}$ and the functions $\bar{u}$ and $\bar{m}$, $\T^{d}$ is the $d$-dimensional torus, $H$ is the Hamiltonian and $F$ the coupling, both related to the structure of the cost, and $H_{p}$ is the partial derivative of $H$ with respect to the second variable. The solution of the first equation in \eqref{EMFG-intro} can be interpreted as the equilibrium value function of a ``small" player whose cost depends on the density $\bar{m}$ of the other players, while the second equation characterizes  the distribution of players at the equilibrium. It is well known (see e.g. \cite{LLMFG, LLMFG1, Feleqi}) that there exists a  solution $(\bar{\la},\bar{u},\bar{m})$ in  $\R\times C^{2}(Q)\times W^{1,p}(Q)$ for all $1\leq p<\infty$ to \eqref{EMFG-intro}, under a wide range of sufficient conditions. Moreover, uniqueness holds under the following \emph{monotonicity condition} on $F$:
\begin{equation}\label{monotonicity-condition}
\forall m,m'\in\mathcal{P}(Q),\qquad \int_{Q}\left( F(x,m)-F(x,m')  \right)\dd(m-m')(x) \geq 0.
\end{equation}
The interpretation of the above monotonicity condition is that the players dislike congested regions and prefer configurations in which they are scattered.

Another well known example of stationary MFG systems, is the case where players aim to minimize a discounted infinite-horizon cost functional. In that case, the MFG system takes the following form  (see e.g. \cite{ Feleqi}, among others):
\begin{equation}
\label{EMFG-intro-discount} 
\left\{
\begin{aligned}
&-\s \D \bar{v}+H(x,D\bar{v})+\rho \bar{v}=F(x,\bar{\mu})\quad \mbox{ in }Q \\
\\
&-\s \D \bar{\mu}-\div(\bar{\mu}H_{p}(x,D\bar{v}))=0 \quad \mbox{ in }Q\\      
\\
&\bar{\mu}\geq 0,\quad <\bar{\mu}>=1\quad 
\end{aligned}
\right.
\end{equation} 
where $\rho >0$. It is also well known (see \cite{BardiAMS, Arisawa, Feleqi}) that, under several technical conditions on $H$ and $F$,  there exists a solution $(\bar{v},\bar{\mu})\in C^{2}(Q)\times W^{1,p}(Q)$ for all $1\leq p<\infty$ to \eqref{EMFG-intro-discount}. Moreover, if $H$ has a linear growth, i.e.
$$
\left| H(x,p) \right| \leq C(1+|p|) 
$$
for some constant $C>0$, system \eqref{EMFG-intro} is obtained as a limit of system \eqref{EMFG-intro-discount} when $\rho \to 0$. More precisely, we have that
\begin{equation}\label{small-discount-intro}
\left(\ \rho <\bar{v}>, \ \bar{v}-<\bar{v}>, \ \bar{\mu}\right) \longrightarrow (\bar{\la}, \bar{u}, \bar{m})\qquad \mbox{ in } \R\times C^{2}(Q)\times L^{\infty}(Q)\quad \mbox{ as }\rho\to 0.
\end{equation}

In this paper, we consider a situation where the evolution of the players is driven by a system of stochastic differential equations, and where
choosing a strategy amounts to choosing a drift vector field that has a suitable regularity; at any time $t\geq 0$, each player seeks to minimize a cost functional which depends on the current state of the system, and on the possible future evolution $(s>t)$ of the player, which is related to her/his choice of a vector field $\a_{t}(.)$ at time $t$.
Thus, choosing the optimal $\a_{t}(.)$ amounts to plan optimally the future evolution of the player, assuming no evolution in her/his environment.
Players follow their planned evolution and adjust their drift $\a_{t}(.)$ according to the observed changes. Further details and explanations about the model will be given in Section \ref{section1bis}.

For the choice of $\a_{t}(.)$
we consider two different cost structures:  a discounted cost functional; and a long-time average cost (see Section \ref{section1bis}). As we already pointed out, the scheduling gives rise to two time scales: a slow time scale ``$t$" linked to the evolution of the state of the system; and a fast scale ``$s$" (which does not appear explicitly in the MFG systems) related to the scheduling.
Under some assumptions on $ H $ and $ F $, we show in Section \ref{section1bis} that at the mean field limit one gets systems of equations of the following form:
\begin{equation}
\label{QSS-intro-2} 
\left\{
\begin{aligned}
&-\s'\D v+H(x,Dv)+\rho v=F(x,\mu(t)) \quad \mbox{ in } (0,T)\times Q\\
\\
&\partial_{t}\mu-\s\D \mu-\div(\mu H_{p}(x,Dv))=0\quad \mbox{ in } (0,T)\times Q \\   
\\
& \mu(0)=m_{0}\geq 0  \quad \mbox{ in } Q , \quad <m_{0}>=1; 
\end{aligned}
\right.
\end{equation}
and
\begin{equation}
\label{QSS-intro-1} 
\left\{
\begin{aligned}
&-\s'\D u+H(x,Du)+\la(t)=F(x,m(t)) \quad \mbox{ in } (0,T)\times Q\\
\\
&\partial_{t}m-\s\D m-\div(m H_{p}(x,Du))=0\quad \mbox{ in } (0,T)\times Q \\   
\\
& m(0)=m_{0}\geq 0  \quad \mbox{ in } Q , \quad <m_{0}>=1,   \quad <u>=0 \quad  \mbox{ in } (0,T), 
\end{aligned}
\right.
\end{equation}
where $ \rho>0$ is fixed throughout this paper, $\s',T>0$, $m_{0}$ is the initial density of players, and all functions are $\Z^{d}$-periodic. Note that $(\la, u)$ (resp. $v$) depends on time only through $m$ (resp. $\mu$). 
The parameters $\s'$ and $\s$ are respectively: the noise level related to the prediction process (the assessment of the future evolution), and the noise level associated to the evolution of the players. 
System \eqref{QSS-intro-1} corresponds to the case of a long time average cost functional, while system \eqref{QSS-intro-2} corresponds to the case of a discounted cost functional. We shall see that for any time $t$, $(\la(t),u(t))$ (resp. $v(t)$) characterizes \emph{a local Nash equilibrium} related to a long time average cost (resp. a discounted cost). The first equations in \eqref{QSS-intro-1} and \eqref{QSS-intro-2} give the ``evolution" of the game value of a ``small" player, and expresse the  adaptation of players choices to the  environment evolution.
The evolution of $\mu$ and $m$ expresses the actual evolution of the population density. We refer to Section \ref{section1bis} for more detailed explanations.

In contrast to most MFG systems, the uniqueness of solutions to systems \eqref{QSS-intro-1} and \eqref{QSS-intro-2} does not require the monotonicity condition \eqref{monotonicity-condition}  nor the convexity of $H$ with respect to the second variable. This fact is essentially related to the forward-forward structure of the systems.  We also show that the small-discount approximation \eqref{small-discount-intro} holds for quasi-stationary models under the same conditions as for the stationary ones.  Under the monotonicity condition \eqref{monotonicity-condition}, we prove in Section \ref{section-exp-convergence} that for a quadratic Hamiltonian, 
a solution $(\la, u,m)$ to \eqref{QSS-intro-1}  converges exponentially fast in some sense to the unique equilibrium $(\bar{\la}, \bar{u}, \bar{m})$ of \eqref{EMFG-intro} as $t\to +\infty$, provided that $m_{0}-\bar{m}$ is sufficiently small and $\s=\s'$. An analogous result holds also for systems \eqref{EMFG-intro-discount}-\eqref{QSS-intro-2} when the discount rate $\rho$ is small enough. 
This asymptotic behavior is interpreted by the emergence of a \emph{self-organizing phenomenon} and \emph{a phase transition} in the system. 
Note that this entails in particular that our systems can exhibit a large scale structure even if the cohesion between the agents is only maintained by interactions between neighbors.
The techniques used to prove this asymptotic results rely on some algebraic properties pointed out in \cite{Cardaliaguet1} specific to the quadratic Hamiltonian. On the other hand, one can not use the usual duality arguments  to show convergence for general data.
Therefore the convergence remains an open problem for more general cases.

Similar asymptotic results were established for the MFG system in \cite{Cardaliaguet1, Cardaliaguet3} for local and non local coupling. 
Long time convergence of forward-forward MFG models is also discussed in \cite{Gomes, Achdou}.
Self-organizing and phase transition in Mean Field Games were addressed in \cite{Mehta1,Mehta2, Mehta3}, for applications in neuroscience, biology, economics, and engineering.
For an overview on collective motions and self-organization phenomena in mean field models, we refer to \cite{Degond4} and the references therein.
 The derivation of the Mean Field Games system was addressed in \cite{LLMFG, LLMFG1, Feleqi2} for the ergodic case (long time average cost). More general cases were analyzed in the important recent paper \cite{master-equation} on the master equation and its application to the convergence problem in Mean Field Games.
The reader will notice in Section \ref{section1bis} that the analysis of the mean-field limit in our case is very similar to that of the McKean-Vlasov equation. Therefore the proof of convergence is less technical than in 
\cite{master-equation} and is based on the usual coupling arguments (see e.g. \cite{Snitzman, Mckean, Méléard}, among others). 
MFG models with myopic players are briefly addressed in \cite{Achdou} for applications to urban settlements and residential choice. However, the sense given to ``myopic players" is different from the one we are considering in this paper: indeed, ``myopic players'' in \cite{Achdou} corresponds to individuals which compute their cost functional taking only into account their very close neighbors, while in this paper "myopic players" refers to individuals which anticipate nothing and only undergo the evolution of their environement. 
In \cite{Cristiani}, the authors introduce a model for the study of crowds dynamics, that is very similar to the one addressed in this paper: in Section 2.2.2, the authors consider a situation where at any time pedestrians build the optimal path to destination, based on the observed state of the system. Although the approaches are different, the two models have many similarities.

Local Nash equilibria for mean field systems of rational agents were also considered in  \cite{Degond1, Degond2, Degond3}. The authors use the ``Best Reply Strategy approach" to derive a kinetic equation and provide applications to the evolution of wealth distribution in a conservative \cite{Degond2} and non-conservative \cite{Degond3} economy. The link between Mean Field Games and the ``Best Reply Strategy approach" is analyzed in \cite{Degond0}.

The paper is organized as follows: In Section \ref{section2}, we give sufficient conditions for the existence and uniqueness of classical solutions for systems  \eqref{QSS-intro-2} and \eqref{QSS-intro-1}. The proofs rely on continuous dependence estimates for Hamilton-Jacobi-Bellman equations \cite{Marchi}, the small-discount approximation,  and the non-local coupling which provides compactness and regularity. 
Section \ref{section1bis} is devoted to a detailed derivation of systems \eqref{QSS-intro-2} and \eqref{QSS-intro-1} from  $N$-players stochastic differential games models. 
In Section \ref{section-exp-convergence}, we prove the exponential convergence result for system \eqref{QSS-intro-1}. Finally, the Appendix recall some elementary facts on the Fokker-planck equation. 

\subsection*{Notations and assumptions} 
For simplicity, we work in a periodic setting in order to avoid issues related to boundary conditions or conditions at infinity. Therefore we will often consider functions as defined on $Q:=\T^{d}$ (the $d$-dimensional torus). Throughout the paper, $d\geq 1$,  and the usual inner product on $\R^{d}$ is denoted by $x.y$ or $<x,y>$.
We  use the notation $<f>:=\int_{Q}f$ and denote $\P(Q)$ the set of probability measures on $Q$. Recall that $\P(Q)$ becomes a compact topological space when endowed with the $\weak^{\ast}$-convergence thanks to Prokhorov's theorem. Moreover,  this topology is metrizable, e.g. by the Kantorowich-Rubinstein distance:
$$
\ddd_{1}(\pi,\pi') := \sup\left\{ \int_{Q}f(x) \dd (\pi-\pi')(x), \mbox{ where } f:Q\rightarrow\R \mbox{ is }1\mbox{-Lipschitz continuous}   \right\}.
$$
We denote by $C(Q)$ the set of $\Z^{d}$-periodic continuous functions on $\R^{d}$, by $C^{m+\ell}(Q)$, $m\in \N$, $\ell \in (0,1]$, the set of $\Z^{d}$-periodic functions having $m$-th order derivatives which are $\ell$-H\"older continuous, by $\left(L^{p}(Q) , \| . \| _{p}\right)$, $1\leq p \leq \infty$, the set of $p$-summable Lebesgue measurable and $\Z^{d}$-periodic functions on $Q$, by $W^{k,p}(Q)$,  $k\in \N$, $1\leq p \leq \infty$, the Sobolev space of $\Z^{d}$-periodic functions  having a weak derivatives up to order $k$ which are $p$-summable on $Q$,
and for a given Lipschitz continuous function $f$, we define
$$
\| f \|_{\Lip} := \sup_{x\neq y} \frac{\left\vert f(x)-f(y)\right\vert}{\vert x- y\vert}.
$$
For $\gamma\in(0,1)$, we use the notation $C^{1+\gamma/2,2+\gamma}$  for parabolic H\"older spaces, 
with the norm $\| .\|_{C^{1+\gamma/2,2+\gamma}}$,
as defined in \cite{Parabolic67}.  For a given $T>0$, we note
$$
Q_{T}:=(0,T)\times Q, \quad \mbox{ and }\quad \overline{Q_{T}} := [0,T]\times Q,
$$
and $C^{1+\gamma/2,2+\gamma}(\overline{Q_{T}})$ 
the space of $\Z^{d}$-periodic function in $C^{1+\gamma/2,2+\gamma}([0,T]\times \R^{d})$.

Throughout the proofs, we denote by $C$  a generic constant, and we use the notation $C_{\eta,\b}$ to point out the dependence of the constant on parameters $\eta,\b$. For any vector $(y_{j})_{1\leq j\leq N}$ we use the notation $y^{-i}=(y_{j})_{j\neq i}$. For a random variable $\mathbb{X}$ defined on some probability space, $\mathcal{L}(\mathbb{X})$ denotes the law of $\mathbb{X}$. \emph{Throughout the paper, $\gamma\in(0,1)$ is a fixed parameter.} 

\subsection*{Acknowledgement} This work was supported  by LABEX MILYON (ANR-10-LABX-0070) of Universit{\'e} de Lyon, within the program "Investissements d'Avenir" (ANR-11-IDEX-0007) operated by the French National Research Agency (ANR), and partially supported
by project (ANR-16-CE40-0015-01) on Mean Field Games.

The author would like to thank Martino Bardi and Pierre Cardaliaguet for fruitful discussions.

\section{Analysis of the quasi-stationary systems}\label{section2}

This section is devoted to the analysis of systems \eqref{QSS-intro-2} and \eqref{QSS-intro-1}. A detailed derivation of these systems from a $N$-person differential game will be given in Section \ref{section1bis}. 

We shall use the following conditions:

\begin{enumerate}[ label=($\mathcal{H}$\arabic*)]
\item\label{A1} 
the operator $m\to F(.,m)$ is defined from $\P(Q)$ into $\Lip (Q):=C^{0+1}(Q)$, and satisfies
\begin{equation} \label{condition2-sur-F}
\sup_{m\in \P(Q)} \| F(.,m)\|_{\Lip} < \infty; 
\end{equation} 

\item \label{A4}
the Hamiltonian $H:Q\times \R^{d} \longrightarrow \R$ is locally Lipschitz continuous, and  $\Z^{d}$-periodic with respect to the first variable; 
\item \label{A3} $H_{p}$ exists and is locally Lipschitz continuous;
\item \label{A5} $D_{x}H_{p}$ and $H_{pp}$ exist and are locally Lipschitz continuous;
\item\label{A2}
there exists a constant $\k_{F}>0$ such that, for any $m,m'\in \P(Q)$, 
$$
\Vert F(.,m)-F(.,m') \Vert_{\infty} \leq \k_{F}\ddd_{1}(m,m');$$
\item \label{A6}
$m_{0}$ is a probability measure, absolutely continuous with respect to the Lebesgue measure, and its density  $m_{0}$ belongs to $C^{2+\gamma}(Q)$. 
\end{enumerate}
The Hamiltonian $H$ satisfies one of the following sets of conditions: 

\begin{enumerate}[ label=$\mathcal{C}$\arabic*.]
\item\label{CH2} $H$ grows at most linearly in $p$, i.e., there exists $\k_{H}>0$, such that
$$
\left| H(x,p) \right| \leq \k_{H}\left(1+ \left| p \right|\right), \quad \forall x\in Q,\  \forall p\in\R^{d};
$$
\item\label{CH} $H$ is superlinear in $p$ uniformly in $x$, i.e.,
$$
\inf_{x\in Q}\vert H(x,p)\vert/\vert p\vert \rightarrow +\infty \quad \mbox{ as } \vert p\vert \rightarrow +\infty,
$$
and there exists $\theta \in(0,1)$, $\k>0$, such that a.e $x\in Q$ and $\vert p\vert$ large enough,
\begin{equation}\label{oscillation-condition}
\left<D_{x}H , p\right>+\theta. H^{2} \geq -\k\vert p\vert^{2}.
\end{equation}
\end{enumerate} 

Condition \eqref{CH2} arises naturally in control theory when the controls are chosen in a bounded set, 
whereas under condition \eqref{CH} the control variable of each player can take any orientation in states space and can be arbitrary large with a large cost. As it is pointed out in \cite{Feleqi, LLMFG, LLMFG1}, the condition \eqref{oscillation-condition} is interpreted as a condition on the oscillations of $H$ and plays no role when $d=1$.

A triplet $(\lambda,u,m)$ is a classical solution to \eqref{QSS-intro-1}, if  $m : [0,T]\times \R^{d}\longrightarrow \R$ is continuous, of class $C^{2}$ in space, and of class $C^{1}$ in time, $u: (0,T)\times \R^{d} \longrightarrow \R$ is of class $C^{2}$ in space, and $(\lambda, u,m)$ satisfies \eqref{QSS-intro-1} in the classical sense. Similarly, a couple $(v,\mu)$ is a classical solution to \eqref{QSS-intro-2}, if  $\mu :  [0,T]\times \R^{d}\longrightarrow \R$ is continuous, of class $C^{2}$ in space,  of class $C^{1}$ in time, $v: (0,T)\times \R^{d}\longrightarrow \R$ is of class $C^{2}$ in space and $(v,\mu)$ satisfies \eqref{QSS-intro-2} in the classical sense. 

In this section, we give an existence and uniqueness result of classical solutions for system \eqref{QSS-intro-2} and \eqref{QSS-intro-1} under condition \eqref{CH2}.
In addition, we show that system \eqref{QSS-intro-1} is also well-posed under condition \eqref{CH}.
 
We start by dealing with the case where the Hamiltonian has a linear growth (condition \eqref{CH2}). 
Let us consider the \emph{quasi-stationary approximate problem} \eqref{QSS-intro-2}.
We start by analyzing the first equation in \eqref{QSS-intro-2}. 
\begin{lemma}\label{lemma-nouveausystem}
Under assumptions \ref{A1}, \ref{A4} and \eqref{CH2}, for any $\mu \in \P(Q)$ and $\varrho>0$, the problem 
\begin{equation}\label{equation-Appro}
-\s'\D v+H(x,Dv)+\varrho v=F(x,\mu) \quad \mbox{ in } Q
\end{equation}
\begin{subequations}
has a unique solution $v_{\varrho}[\mu]\in C^{2+\gamma}(Q)$. In addition, there exists constants $\k_{0}>0$ and $\theta \in(0,1]$, such that for any $\mu \in \P(Q)$ and $\varrho >0$, the following estimates hold
\begin{equation}\label{crucial-point-approximation-anx}
\left\| \varrho v_{\varrho}[\mu] \right\|_{\infty} \leq \| F\|_{\infty}+\k_{H},
\end{equation}
\begin{equation}\label{crucial-point-approximation}
\| v_{\varrho}[\mu]-\left< v_{\varrho}[\mu] \right> \|_{C^{2+\theta}} \leq \k_{0}.
\end{equation}
\end{subequations}
\end{lemma}
\begin{proof}
The proof of existence and uniqueness for equation \eqref{equation-Appro}  relies on regularity results and a priori estimates from elliptic theory. A detailed proof to this result is given in  \cite[Theorem 2.6]{Feleqi} in a more general framework. By looking at the extrema of $v[\mu]$, one easily gets \eqref{crucial-point-approximation-anx}. The second bound is proved by contradiction using the strong maximum principle. The details of the proof are given in \cite[Theorem 2.5]{Feleqi}. Condition \eqref{condition2-sur-F} ensures that the constant $\k_{0}$ does not depend on $\mu$.
\end{proof}
\begin{remark}
Note that the well-posedness of equation \eqref{equation-Appro} still holds under the following condition on $H$ (the so-called  \emph{natural growth condition}),
\begin{equation}\nonumber
\exists \k'_{H}>0, \qquad \left| H(x,p) \right| \leq \k'_{H}\left(1+ \left| p \right|^{2}\right), \quad \forall x\in Q,\forall p\in\R^{d};
\end{equation}
which is less restrictive than \eqref{CH2}. 
\end{remark}

We now state  a continuous dependence estimate due to Marchi  \cite{Marchi}, which plays a crucial role.
\begin{lemma}\label{detailed-continuous-dependence}
Assume \ref{A1}-\ref{A3}, and \eqref{CH2}. For any $\mu,\mu' \in \P(Q)$ and $\varrho>0$, we have that
\begin{subequations}
\begin{equation}\label{detailed-continuous-dependence-1}
\| v_{\varrho}[\mu] -v_{\varrho}[\mu']\|_{\infty} \leq  \varrho^{-1}\| F(.,\mu) -F(.,\mu')\|_{\infty}.
\end{equation}
Moreover, for any $M>0$, there exists a constant $\chi_{M}>0$, such that for any $\varrho \in (0,M)$ and $\mu,\mu' \in \P(Q)$, the following holds
\begin{equation}\label{detailed-continuous-dependence-2}
\| w_{\varrho}[\mu] -w_{\varrho}[\mu']\|_{C^{2}} \leq \chi_{M} \| F(.,\mu) -F(.,\mu')\|_{\infty}, 
\end{equation}
\end{subequations}
where $w_{\varrho}= v_{\varrho}-\left< v_{\varrho} \right>$.
\end{lemma}
\begin{proof}
Note that 
$$
v^{\pm}:= v_{\varrho}[\mu'] \pm \varrho^{-1} \| F(.,\mu) -F(.,\mu')\|_{\infty},
$$
are respectively a super- and a sub-solution to equation \eqref{equation-Appro} with the coupling term $F(.,\mu)$. Thus, estimate \eqref{detailed-continuous-dependence-1} follows thanks to the comparison principle. 

The proof of \eqref{detailed-continuous-dependence-2} is similar to \cite[Theorem 2.2]{Marchi}. Nevertheless we give a proof to this result because in this particular framework we do not need to fulfill all the conditions of \cite{Marchi}. We shall proceed by contradiction assuming that there exists sequences $(\varrho_{n}) \in (0,M)$, $(\mu_{n}), (\mu'_{n}) \in \P(Q)$, such that for any $n\geq 0$,
\begin{equation}\label{eqtim-contracdiction}
c_{n} \geq n  \| F(.,\mu_{n}) -F(.,\mu'_{n})\|_{\infty},
\end{equation}
where $c_{n}:=\| w_{\varrho_{n}}[\mu_{n}] - w_{\varrho_{n}}[\mu'_{n}] \|_{C^{2}}$, and $\lim_{n}\varrho_{n}=0$. Note that the function
$$W_{n}:=  c_{n}^{-1}\left( w_{\varrho_{n}}[\mu_{n}] - w_{\varrho_{n}}[\mu'_{n}] \right)$$
satisfies the following equation
$$
R_{n} -\s'\D W_{n} + f_{n}. DW_{n}  =0,
$$
where 
$$
R_{n}:=\varrho_{n} W_{n}+c_{n}^{-1}\varrho_{n} \left<v_{\varrho_{n}}[\mu_{n}] - v_{\varrho_{n}}[\mu'_{n}]\right>+c_{n}^{-1}\left( F(.,\mu_{n}) -F(.,\mu'_{n})   \right),
$$
and 
$$
f_{n}(x) := \int_{0}^{1} H_{p}\left(x,s D w_{\varrho_{n}}[\mu_{n}]+(1-s) D w_{\varrho_{n}}[\mu'_{n}] \right) \dd s.
$$
Using \eqref{eqtim-contracdiction} and \eqref{detailed-continuous-dependence-1}, one checks that
$
\lim_{n}\| R_{n}\|_{\infty} =0.
$
In addition,  \ref{A3} and \eqref{crucial-point-approximation} entails that $\| f_{n} \|_{\Lip}$ is uniformly bounded. Moreover, invoking standard regularity theory for linear elliptic equations (see e.g. \cite{Gilbarg}), the sequence $(W_{n})$ is uniformly bounded in $C^{2+\theta'}(Q)$ for some $\theta'\in (0,1]$. We infer that $(f_{n} ,W_{n})$ converge uniformly to some $(f,W)$ in $C(Q)\times C^{2}(Q)$ which satisfies
$$
-\s'\D W + f. DW  =0, \quad \| W \|_{C^{2}}=1, \quad \left< W \right>=0.
$$
Since $W$ is periodic, we deduce from the strong maximum principle that $W$ must be constant; this provides the desired contradiction.
\end{proof}

\begin{corollary}
\begin{subequations}
Under \ref{A2} and assumptions of Lemma \ref{detailed-continuous-dependence},
for any $\varrho>0$ there exists a constant $\k_{\varrho}>0$ such that,
\begin{equation}\label{continuous-newsystem}
\| v_{\varrho}[\mu] -v_{\varrho}[\mu'] \|_{C^{2}} \leq \k_{\varrho} \ddd_{1}(\mu,\mu')
\end{equation}
for any $\mu,\mu' \in \P(Q)$.
\end{subequations}
\end{corollary}

We shall give now an existence and uniqueness result for system \eqref{QSS-intro-2}.  
 \begin{theorem}\label{prop-existence-newsyst}
Under conditions \ref{A1}-\ref{A6} and \eqref{CH2}, there exists a unique classical solution $(v, \mu)$ in $C^{1/2}\left([0,T] ; C^{2}(Q)\right)\times C^{1,2}(\overline{Q_{T}})$ to the problem \eqref{QSS-intro-2}.
\end{theorem}
\begin{proof}
\emph{Existence :} For a constant $\d>0$ large enough to be chosen below, let $\Cc_{\d}$ be the set of maps $\mu\in C([0,T];\P(Q))$ such that 
$$
\sup_{s\neq t} \frac{\ddd_{1}(\mu(t),\mu(s))}{\vert t-s\vert^{1/2}} <\d.
$$
Note that $\Cc_{\d}$ is compact thanks to Ascoli's Theorem, and the compactness of $(\P(Q),\ddd_{1})$. 
We aim to prove our claim using Schauder's fixed point theorem (see e.g. \cite[p. 25]{Fixedpoint}). Set for any $(x,\nu)\in Q\times \P(Q)$,
$$
G(x;\nu):=H_{p}(x,Dv_{\rho}[\nu](x)).
$$ 
Note that  $G$ and $D_{x}G $ are uniformly bounded thanks to \ref{A3}, \ref{A5}, and the uniform bound
\eqref{crucial-point-approximation}.
We define an operator 
$$
\Psi : \Cc_{\delta} \rightarrow \Cc_{\delta},
$$
such that, for a given $\nu \in \Cc_{\delta}$,  $\Psi\nu:=\mu$ is the solution to the following ``McKean-Vlasov" equation
\begin{equation}\label{prob-anx-newsyst}
\p_{t}\mu-\s\D \mu -\div(\mu G(x;\nu(t)))=0, \qquad \mu(0)=m_{0}.
\end{equation}

Let us check that $\Psi$ is well defined. Note that the above equation can be written as 
$$
\p_{t}\mu-\s\D \mu -\langle D\mu, G(x;\nu(t)) \rangle-\mu \div\left( G(x;\nu(t)) \right)=0.
$$
Using assumption \ref{A3} and estimate \eqref{continuous-newsystem}, we have for any $t\neq s$  and $x\in Q$
$$
\left\vert G(x;\nu(t))-G(x;\nu(s)) \right\vert \leq C_{\rho, H_{p}}\ddd_{1}(\nu(t),\nu(s)),
$$
so that
$$
\sup_{x\in Q}\sup_{t\neq s} \frac{\vert G(x;\nu(t))-G(x;\nu(s)) \vert }{|t-s |^{1/2}}  \leq C_{\d,\rho, H_p}<\infty.
$$
In the same way, one checks that functions $(x,t)\rightarrow G(x,\nu(t))$ and $(x,t)\rightarrow \div\left[ G(x,\nu(t))\right]$ are in $C^{\gamma'/2,\gamma'}(\overline{Q_{T}})$, where $\gamma'=\min(\gamma, \theta)$, thanks to Lemma \ref{lemma-nouveausystem} and \ref{A5}. 
Here $\gamma$ and $\theta$ are the H\"older exponents appearing in \ref{A6} and \eqref{crucial-point-approximation} respectively. 
We infer that problem \eqref{prob-anx-newsyst} has a unique solution $\mu\in C^{1+\gamma'/2,2+\gamma'}(\overline{Q_{T}})$ which satisfies
\begin{equation}\label{estim-mesure-newsyst}
\Vert \mu\Vert_{C^{1+\gamma'/2,2+\gamma'}} \leq C_{\| G \|_{C^{1}}} \Vert m_{0}\Vert_{C^{2+\gamma'}},
\end{equation}
owing to existence and uniqueness theory for parabolic equations in H\"older spaces
\cite[Theorem IV.5.1 p. 320]{Parabolic67}. Furthermore, using classical properties of Fokker-Planck equation (see Lemma \ref{FKP-Lemma}), it follows that 
$$
\ddd_{1}(\mu(t),\mu(s))\leq C_{T}\left(1+\Vert G\Vert_{\infty}\right)\vert t-s\vert^{1/2}.
$$
Therefore $\mu \in \Cc_{\d}$ for big enough $\d$, since $\Vert G\Vert_{\infty}$ and $C_{T}$ does not dependent on $\nu$ nor on $\d$. 
In particular, the operator $\Psi$ is well defined form $\Cc_{\d}$ into $\Cc_{\d}\cap C^{1+\gamma'/2,2+\gamma'}(\overline{Q_{T}})$.

Let us check now that $\Psi$ is continuous. Given a sequence $\nu_{n}\rightarrow \nu$ in $\Cc_{\d}$, let 
$$\mu_{n}:=\Psi\nu_{n},\quad \mbox{ and } \quad 
\mu:=\Psi \nu.$$ 
Invoking Ascoli's Theorem, estimate \eqref{estim-mesure-newsyst} and the uniqueness of the solution to  \eqref{prob-anx-newsyst}, it holds that
$$
\lim_{n}\Vert \mu_{n}-\mu\Vert_{C^{1,2}}= 0.
$$ 
The convergence is then easily proved to be in $C([0,T],\P(Q))$. Thus, by Schauder fixed point theorem the map $\Psi: \Cc_{\delta}\rightarrow \Cc_{\delta}$ has a fixed point $\mu \in C^{1,2}(\overline{Q_{T}})$ and $(v_{\rho}[\mu], \mu)$ is a classical solution to \eqref{QSS-intro-2}. In addition, estimate \eqref{continuous-newsystem} entails that 
 $v_{\rho}[\mu] \in C^{1/2}\left( [0,T] ; C^{2}(Q)\right)$.

\emph{Uniqueness :} Let $(v,\mu)$ and $(v',\mu')$ be two solutions to  the system \eqref{QSS-intro-2}, $w:=v-v'$ and $\nu:=\mu-\mu'$. One checks that
\begin{equation}
\nonumber
\left\{
\begin{aligned}
&\partial_{t}\nu-\s\D \nu-\div (\nu H_{p}(x,Dv))=\div \left(\mu'\left(H_{p}(x,Dv)-H_{p}(x,Dv')\right)\right),\\
&\nu_{t=0}=0.
\end{aligned}
\right.
\end{equation}
By standard duality techniques, we deduce that
$$
\frac{1}{2}\frac{\dd}{\dd t}\Vert \nu(t)\Vert_{2}^{2}+\s\|  D\nu(t)\|_{2}^{2} = -\left(\int_{Q} \nu  D\nu.H_{p}(x,Dv)+\int_{Q} \mu'D\nu.\left(H_{p}(x,Dv)-H_{p}(x,Dv')\right) \right), 
$$
so that
\begin{equation}\label{estim-corr-prf-th21-newsyst}
\frac{1}{2}\frac{\dd}{\dd t}\Vert \nu(t)\Vert_{2}^{2}+\frac{\s}{2}\| D \nu(t)\|_{2}^{2} \leq C\left(\Vert G \Vert_{\infty}   \| \nu(t)\|_{2}^{2}+\int_{Q} \vert \mu' Dw\vert^{2}\right).
\end{equation}
From \eqref{continuous-newsystem} and \eqref{estim-mesure-newsyst}, we infer that
$$
\int_{Q} \vert \mu' Dw\vert^{2}  \leq C\| m_{0}\|_{C^{2+\gamma}}^{2}\ddd_{1}(\mu,\mu')^{2} \leq C\| m_{0}\|_{C^{2+\gamma}}^{2} \| \nu(t) \|_{2}^{2}.
$$
Plugging this into \eqref{estim-corr-prf-th21-newsyst} provides
$$
\| \nu(t) \|_{2}^{2} \leq C\int_{0}^{t} \| \nu (s) \|_{2}^{2}\dd s,
$$
which implies that $\nu\equiv0$,  and so $w\equiv0$ thanks to \eqref{continuous-newsystem}. The proof is complete.
\end{proof}

Let us now deal with system \eqref{QSS-intro-1}. We shall start by proving the well-posedness  for the first equation in \eqref{QSS-intro-1} and by
 giving a continuous dependence estimate.

\begin{lemma}\label{Lemma-interminable}
Under conditions \ref{A1}-\ref{A3}, \ref{A2} and \eqref{CH2}, for any measure $m\in \P(Q)$, there exists a unique solution $(\la[m], u[m]) \in \R \times C^{2}(Q)$ to the problem
\begin{equation}\label{ergodicproblem-interminable}
-\s'\D u + H(x,Du) +\la = F(x,m)\quad \mbox{ in }Q, \quad <u>=0.
\end{equation}
\begin{subequations}
Moreover, for any $m,m' \in \P(Q)$, the following estimates hold
\begin{equation}\label{ergodicproblem-interminable-1}
\left| \la[m]-\la[m'] \right| \leq \k_{F}\ddd_{1}(m,m'),
\end{equation}
\begin{equation}\label{ergodicproblem-interminable-2}
\left\| u[m]-u[m'] \right\|_{C^{2}} \leq \chi_{1}\k_{F} \ddd_{1}(m,m').
\end{equation}
\end{subequations}
\end{lemma}
\begin{proof}
It is well known (see e.g. \cite{BardiAMS, Arisawa}) that for a given $m\in \P(Q)$, there exists a unique periodic solution $(\lambda[m],u[m])$ in $\R\times C(Q)$ to \eqref{ergodicproblem-interminable}. Regularity of the solution, and estimates \eqref{ergodicproblem-interminable-1}, \eqref{ergodicproblem-interminable-2} follow from Lemma \ref{detailed-continuous-dependence}, and small-discount approximation techniques (see e.g. \cite{BardiAMS, Arisawa, Feleqi}).
\end{proof}

\begin{remark}
It is possible to show more regularity for the maps $m\to \la[m]$, $m\to u[m]$ under additional regularity assumptions on $F$ and $H$. For instance,  if $H_{pp}\geq \k_{e} I_{d}$ for some $\k_{e}>0$, and $F$ satisfies 
$$
 \sup_{m\neq m' }\ddd_{1}(m,m')^{-1}\left\| \frac{\d F}{\d m}(.,m,.)-\frac{\d F}{\d m}(.,m',.) \right\|_{C^{0}\times C^{1}} < \infty,
$$
then $u[.]$ and $\la[.]$ are of class $C^{1}$ in $\P(Q)$. We refer to \cite{master-equation} for the definition of derivatives in $\P(Q)$ and notations.
In addition, we have that
$$\frac{\d u}{\d m}(m)(\nu):=w(m,\nu) \quad \mbox{ and }\quad \frac{\d \la}{\d m}(m)(\nu):=\varsigma(m,\nu)$$
for any $m$ in $\P(Q)$ and any signed measure $\nu$ on $Q$, where $(\varsigma,w)$ is the solution to the following problem
$$
-\s'\D w+H_{p}\left(x,Du[m]\right).Dw+\varsigma=\frac{ \d F}{\d m}(m)(\nu) \quad \mbox{ in } Q, \ \mbox{ and } \ \left<w\right>=0.
$$
One has also an analogous result for the map $v_{\varrho}[.]$ defined in Lemma \ref{lemma-nouveausystem}.
We omit the details and invoke \cite[Proposition 3.8]{master-equation} for a similar approach. 
\end{remark}

We prove now well-posedness for system \eqref{QSS-intro-1}.
\begin{theorem}\label{existence-QSS-a}
Under assumptions \ref{A1}-\ref{A6} and \eqref{CH2}, there exists a unique classical solution 
$(\la, u, m)$ in $C^{1/2}\left([0,T]\right)\times C^{1/2}\left([0,T] ; C^{2}(Q)\right) \times C^{1,2}\left(\overline{Q_{T}}\right)$
to system \eqref{QSS-intro-1}. 
This result holds if one replaces condition \eqref{CH2} by condition \eqref{CH}.
\end{theorem}
\begin{proof} 
The proof of existence relies on small-discount approximation techniques. We give here an adaptation of these techniques for the quasi-stationary case. The crucial point in this proof is estimates \eqref{crucial-point-approximation-anx} and \eqref{crucial-point-approximation}.

Assume first that $H$ satisfies condition \eqref{CH2}.
Let $(v^{\varrho},\mu^{\varrho})$ be the unique classical solution to \eqref{QSS-intro-2} replacing $\rho$ by $\varrho$, and set  $w^{\varrho} := v^{\varrho}- \left<v^{\varrho}\right>$. Invoking 
\eqref{crucial-point-approximation-anx} and \eqref{crucial-point-approximation}, we have 
\begin{equation}
\label{estim-convg-appro-th25-2}
\left\{
\begin{aligned}
&-\s'\D w^{\varrho}+H(x,Dw^{\varrho})+\varrho w^{\varrho}=F(x,\mu^{\varrho}(t))-\varrho <v^{\varrho}> \quad \mbox{ in } Q_{T}, \\
\\
&\partial_{t}\mu^{\varrho}-\s\D \mu^{\varrho}-\div\left(H_{p}(x,Dv^{\varrho})\mu^{\varrho}\right)=0 \quad \mbox{ in } Q_{T}, \quad \mu^{\varrho}(0)=m_{0} \quad \mbox{ in } Q,   \\
\\
&\sup_{0\leq t\leq T}\left\| w^{\varrho}(t)\right\|_{C^{2+\theta}} \leq \k_{0}, \quad  \sup_{0\leq t\leq T}\left\| \varrho v^{\varrho}[\mu(t)] \right\|_{\infty} \leq \| F\|_{\infty}+\k_{H}.\\
\end{aligned}
\right.
\end{equation}
On the other hand, recall that according to \cite[Theorem IV.5.1 p. 320]{Parabolic67} it holds that
\begin{equation}\label{estim-convg-appro-th25-3}
\Vert \mu^{\varrho} \Vert_{C^{1+\gamma'/2,2+\gamma'}} \leq C_{1} \Vert m_{0}\Vert_{C^{2+\gamma'}},
\end{equation}
where $\gamma'=\min(\gamma,\theta)$, and the constant
$C_{1}>0$ is independent of $\varrho$ thanks to \eqref{crucial-point-approximation}. Hence, one can extract a subsequence $\varrho_{n}\to 0$ such that for any $t\in [0,T]$
\begin{equation}\label{convergenceresult-newsyst-th25}
\left( \varrho_{n}\left<v^{\varrho_{n}}(t)\right>, w^{\varrho_{n}}(t), \mu^{\varrho_{n}} \right) \to
\left(  \la(t), u(t), m\right) \quad \mbox{ in }  \R\times C^{2}(Q) \times C^{1,2}(\overline{Q_{T}}) \quad \mbox{ as } n\to\infty,
\end{equation}
where $( \la, u, m)$ is a classical solution to \eqref{QSS-intro-1}. In addition, for any $t,s \in[0,T]$, estimates \eqref{ergodicproblem-interminable-1} and \eqref{ergodicproblem-interminable-2} provide
\begin{equation}\nonumber
\| u[m(t)]-u[m(s)]\|_{C^{2}} \leq \chi_{1}\k_{F} \ddd_{1}(m(t),m(s)),
\end{equation}
and
\begin{equation}\nonumber
\left| \la[m(t)]-\la[m(s)] \right| \leq \k_{F}\ddd_{1}(m(t),m(s)).
\end{equation}
Thus, $u[m] \in C^{1/2}\left( [0,T] ; C^{2}(Q)\right)$ and $\la[m] \in C^{1/2}\left([0,T]\right)$.
The proof of uniqueness is  identical to Theorem \ref{prop-existence-newsyst}. Hence, the proof of well-posedness under \eqref{CH2} is complete.

If we suppose that $H$ satisfies only \eqref{CH}, by virtue of \eqref{condition2-sur-F}  one can derive the following uniform bound using Bernstein's method (see  \cite{LLMFG,LLMFG1} and \cite[Theorem 2.1]{Feleqi}): 
\begin{equation}\label{Bernestein-bound}
\exists \k_{B}>0,\ \forall \nu \in \P(Q)^{N}, \qquad \left\| Du[\nu]\right\|_{\infty} \leq \k_{B}.
\end{equation}
Thus, by a suitable truncation of $H$ one reduces the problem to the previous case.
\end{proof}

\begin{remark}
All the results of this section hold true if one replaces the elliptic parts of the equations
with a more general operator $L$ of the following form:
$$
L := -\Tr\left(\psi(x)D^{2}\right),
$$
where $\psi$ is $\Z^{d}$-periodic,  $\| \psi\|_{\Lip}<\infty$, and there exists $\k_{\psi}>0$ such that $\psi(x)\geq \k_{\psi} I_{d}$.
\end{remark}

\section{Models explanation \& mean field limit}\label{section1bis}
We provide in this section a rigorous interpretation for the quasi-stationary systems \eqref{QSS-intro-2} and \eqref{QSS-intro-1} in terms of $N$-players stochastic differential games. We shall start by writing 
systems of equations for $N$ players, then we pass to the limit when the number of players goes to infinity assuming that all the players are identical. Throughout this section, we employ the notations introduced in Lemma \ref{lemma-nouveausystem} and Lemma \ref{Lemma-interminable}.

\subsection{Stochastic differential games models for $N$-players. } 
We consider a game of $N$-players where at each time agents choose their strategy 
\begin{itemize}
\item[-]  assuming no evolution in their environment; 
\item[-] according to an evaluation of their future situation emanating from the choice.
\end{itemize}
Observing the evolution of the system, players adjust their strategies without anticipating.
More precisely, each player observe
the state of the system at time $t$ and chooses the best drift vector field $\a_{t}(.)$ which optimize her/his future evolution $(s>t)$. The player adapts and corrects her/his choice as the system evolves. This situation amounts to resolving at each moment an optimization problem which consists in finding the vector field (strategy) which guarantees the best future cost.  Our agents are myopic: they anticipate no evolution and only undergo changes in their environment.

Let us now give a mathematical formalism to our model. Let $(W^{j})_{1\leq j \leq N}$ be  a family of $N$ independent Brownian motions in $\R^d$ over some probability space $(\Omega, \mathcal{F},\mathbb{P})$, and $(D^{i})_{1\leq i\leq N}$ be closed subsets of $\R^{d}$. We suppose that the probability space $(\Omega, \mathcal{F},\mathbb{P})$ is rich enough to fulfill the assumptions that will be formulated in this section.
Let $V:=(V^{1},...,V^{N})$ be  a vector of i.i.d random variables with values in $\R^{d}$ that are independent of $(W^{j})_{1\leq j \leq N}$ and let 
$$
\F_{t}:= \sigma\left\{V^{j}, W_{u}^{j},\ \ 1\leq j\leq N, \quad u\leq t \right\}
$$ 
be the information available to the players at time $t$. We suppose that $\F_{t}$ contains the $\mathbb{P}$-negligible sets of $\F$.

Consider a system driven by the following stochastic differential equations
\begin{equation}\label{learning-eq1,1-Cinfo}
\dd X_{t}^{i}= \a_{t}^{i}(X_{t}^{i}) \dd t +\sqrt{2\s_{i}}\dd W_{t}^{i}, \quad X_{0}^{i}=V^{i}, \quad i=1,...,N.
\end{equation}
For any $t\geq 0$, the $i$-th player choses $\a_{t}^{i}$ in the set of \emph{admissible strategies} denoted by $\A^{i}$, that is,
the set of $\Z^{d}$-periodic processes $\a^{i}$ defined on $\Omega$, indexed by $\R^{d}$ with values in $D^{i}$, such that 
\begin{equation}\label{well-posedness-fast-scale}
\sup_{\omega\in\Omega}\| \a^{i}(\omega,.)\|_{\Lip}<\infty.
\end{equation}
The reason of considering condition \eqref{well-posedness-fast-scale} will be clear in \eqref{realdyn-eq1,2} below. At each time $t\geq 0$, player $i$ faces an optimization problem for choosing  $\a_{t}^{i}(.)\in \A^{i}$ which insures the best future cost.  We will explain the optimization problem in Section \ref{section3.1.1}.

These instant choices give rise to a global (in time) strategies $(\a_{t}^{1},...,\a_{t}^{N})_{t\geq0}$ which does not necessarily guarantee the well-posedness of equations \eqref{learning-eq1,1-Cinfo} in a suitable sense. Hence we need to introduce the following definitions:
\begin{definition}\label{definition31}
Let $T>0$ and $i=1,...,N$. We say that the $i$-th equation of \eqref{learning-eq1,1-Cinfo} is well-posed on $[0,T]$, if there exists a process $X^{i}$, unique a.s,  with continuous sample paths on $[0,T]$ which satisfies the following properties:
\begin{enumerate}[label=(\roman*)]
\item $(X^{i})_{t\in[0,T]}$ is  $(\F_{t})_{t\in[0,T]}$-adapted;
\item $\mathbb{P}\left[ X_{0}^{i}=V^{i}\right]=1$;
\item $\mathbb{P}\left[ \int_{0}^{t} \left|\a_{s}^{i}(X_{s}^{i}) \right |\dd s<\infty \right]=1 \ \ \forall t\in[0,T]$;
\item for any $t\in[0,T]$, the following holds
$$
X_{t}^{i}=V^{i}+\int_{0}^{t}\a_{s}^{i}(X_{s}^{i})\dd s + \sqrt{2\s_{i}}W_{t}^{i}\quad \mbox{a.s }.
$$
\end{enumerate}
System \eqref{learning-eq1,1-Cinfo} is well-posed if all equations are.
\end{definition}
\begin{definition}
Let $T>0$ and $i=1,...,N$. We say that the global strategy $(\a_{t}^{i})_{t\geq 0}$ is \emph{feasible} on $[0,T]$, if  the $i$-th equation of \eqref{learning-eq1,1-Cinfo} is well-posed on $[0,T]$.
\end{definition}

Note that in contrast to standard optimal control situations, the optimal global strategy is not a solution to a global (in time) optimization problem, but it is the history of all the choices made during the game. The agents plan and correct their plans as the game evolves, and the global strategy is achieved through this process of planning and self-correction. 

\subsubsection{The case of a long time average cost}\label{section3.1.1}
Consider the case where the $i$-th player seeks to minimize the following long time average cost:
\begin{equation}\label{learning-cost-perfect-info}
J^{i}\left(t,V,\a_{t}^{1},...,\a_{t}^{N}\right) := \liminf_{\tau \rightarrow +\infty} \frac{1}{\tau}\E \left[ \int_{t}^{\tau} L^{i}(\X_{s,t}^{i},\a_{t}^{i}(\X_{s,t}^{i}))+F^{i}(\X_{s,t}^{i} ; X_{t}^{-i} ) \dd s \  \Big| \ \mathcal{F}_{t}  \right],
\end{equation}
where $L^{i} : \R^d \times D^{i} \rightarrow \R $ and $F^{i}:\R^d \times \R^{d(N-1)} \rightarrow \R $ are continuous  and $\Z^{d}$- periodic with respect to the first variable.  At any time $t\geq 0$, the process $(\X_{s,t}^{i})_{s>t}$ represents 
the possible future trajectory of player $i$, related to the chosen strategy (vector field) $\a_{t}^{i}\in \A^{i}$. In other words, $(\X_{s,t}^{i})_{s>t}$  is what is likely to happen (in the future $s>t$) if player $i$ plays $\a_{t}^{i}$ at the instant $t$.
Mathematically, we consider that $(\X_{s,t}^{i})_{s>t}$  are driven by the following (fictitious) stochastic differential equations
\begin{equation}\label{realdyn-eq1,2}
\left\{
\begin{aligned}
&\dd \X_{s,t}^{i}= \a_{t}^{i}(\X_{s,t}^{i})  \dd s +\sqrt{2\s'_{i}}\dd \B_{s-t,t}^{i} \quad s>t,\\
&\X_{t,t}^{i}=X_{t}^{i}, \quad i=1,...,N,
\end{aligned}
\right.
\end{equation}
where  $\{ (\B_{.,t}^{i})_{1\leq i\leq N}\}_{t\geq 0}$ is a family of standard Brownian motions, and for any $t\geq 0$, the process $(\B_{s-t,t}^{i})_{s>t}$ represents the noise related to the future prediction (or guess) of the $i$-th player. For simplicity, we assume that for any $i\in\{1,...,N\}$,  $t\geq0$, and $s>t$, 
\begin{equation}\label{Hypo-indép}
\B_{s-t,t}^{i} \ \mbox{ is independent from } \ \mathcal{F}_{t}.
\end{equation}
Note that system \eqref{realdyn-eq1,2} is well-posed in the strong sense, and that the definition of $(\X_{s,t}^{i})_{s>t}$ introduces a fast (instantaneous) scale `$s$' related to the projection in future, which is different from the real (slow) scale `$t$'. 

The cost functional \eqref{learning-cost-perfect-info} is an evaluation of the future cost of player $i$, given the information available at time $t$. In this model we consider that the evaluation horizon is infinite. The cost structure expresses the fact that agents are myopic: they anticipate no future change and act as if the system will remain immutable. As they adjust, they undergo changes and do not anticipate them.

We now give a definition of Nash equilibrium for our game.
\begin{definition}
We say that a vector of global strategies $(\hat{\a}_t^{1},...,\hat{\a}_t^{N})_{t\geq  0}$ is a \emph{Nash equilibrium} of the $N$-person game on $[0,T]$, for the initial position $V=(V_1,...,V_N)$,  if for any $i=1,...,N$,
$$
(\hat{\a}_{t}^{i})_{t\geq 0} \ \mbox{ is feasible on }[0,T],
$$
and
$$
\hat{\a}_t^{i} =\arg\max_{\a^{i}\in \A^{i}} J^{i}\left(t,V,\hat{\a}_{t}^{1},...,\hat{\a}_{t}^{i-1}, \a^{i},\hat{\a}_{t}^{i+1},...,\hat{\a}_{t}^{N}\right) \quad \mbox{a.s} \quad \forall t\in[0,T].
$$
In other words, a Nash equilibrium on $[0,T]$ is the history of local Nash equilibria, which is feasible on [0,T].
\end{definition}

Next we provide a verification result that produces a Nash equilibrium for the $N$-person game associated to the cost \eqref{learning-cost-perfect-info}. Let us introduce the following notation for empirical measures: 
$$
\hat{\nu}_{Y}^{M} := \frac{1}{M}\sum_{i=1}^{M}\d_{Y_{i}}, \qquad \forall Y=(Y_{i}) \in \R^{Md}.
$$
For any $i=1,...,N$, we suppose that $F^{i}$ depends only on $x\in Q$ and on the empirical density of the other variables. Namely, for any $x\in Q$ and $Y=(Y^{1},...,Y^{N-1})\in \R^{d(N-1)}$,
$$
F^{i}\left(x ; Y \right) := F^{i}\left(x; \ \hat{\nu}_{Y}^{N-1} \right). 
$$
Set 
 for $(x,p)\in Q\times \R^{d}$,
$$
H^{i}(x,p):=\sup_{\a\in D^{i}} \left\{ -p.\a-L^{i}(x,\a) \right\}.
$$
Throughout this section, we assume that assumptions of Theorem \ref{existence-QSS-a} hold for $H^{i}$ and $F^{i}$,
and that the supremum is achieved at a unique point $\bar{\a}^{i}$ in the definition of $H^{i}$, for all $(x,p)$, so that
\begin{equation}\label{optim-prop21-learning}
H_{p}^{i}(x,p)=-\bar{\a}^{i}(x,p):=\arg \max_{\a \in D^{i}} \left\{ -p.\a-L^{i}(x,\a) \right\}.
\end{equation}
We also employ the notations introduced in Lemma \ref{Lemma-interminable}: namely, for any $\pi\in\P(Q)$, we denote by $(\la^{i}[\pi], u^{i}[\pi])$ the unique solution to
\begin{equation}\nonumber
-\s'_{i}\D u^{i} + H^{i}(x,Du^{i}) +\la^{i} = F^{i}(x,\pi)\quad \mbox{ in }Q, \quad <u>=0.
\end{equation}

\begin{remark}
It is possible to consider a more general form for the drift in system \eqref{realdyn-eq1,2}. For instance, one can replace $\a$ by the following (more general) affine form:
$$
f^{i}(x,\a) := g^i(x) + G^i (x)\a,
$$
where $G^i \in \Lip(Q)^{d\times d} $ and $g^i \in \Lip(Q)^d$. Then 
$$
H^i (x,p) = -p.g^i + \sup_{\a\in D^{i}} \left\{ -p.G^i(x) \a-L^{i}(x,\a) \right\}.
$$
If $L^i$ is Lipschitz in $x$, uniformly as $\a$ varies in any bounded subset, and asymptotically super-linear, i.e.
$$
\lim_{|\a| \to +\infty}\inf_{x\in Q} L^i (x,\a)/ |\a| = +\infty,
$$
then the supremum in  the definition of $H^i$ is attained. Uniqueness of the supremum holds if $L^i$ is strictly convex with respect to the second variable.
\end{remark}

The following result characterizes a Nash equilibrium on $[0,T]$ associated to the cost functional \eqref{learning-cost-perfect-info}.

\hfill
\begin{proposition}\label{proposition-learning}
\hfill
\begin{enumerate}[ label=\upshape(\arabic*)]
\item \label{Assert1}
The following system of equations is well-posed on $[0,T]$, 
\begin{equation}\label{Jeux-1-complete-info-1} 
\dd \bar{X}_{t}^{i}= \bar{\a}^{i}\left(\bar{X}_{t}^{i},Du^{i}\left[   \hat{\nu}_{ \bar{X}_{t}^{-i}}^{N-1}\right](\bar{X}_{t}^{i})\right) \dd t  +\sqrt{2\s_{i}}\dd W_{t}^{i},\quad \bar{X}_{0}^{i}=V^{i}, \quad i=1,...,N.\\
\end{equation} 
\item \label{Assert2}
Let for $x\in Q$ and $t\in[0,T]$
$$
\bar{\a}_{t}^{i}(x) := \bar{\a}^{i}\left(x,Du^{i}\left[ \hat{\nu}_{ \bar{X}_{t}^{-i}}^{N-1}\right](x)\right),  \quad i=1,...,N.
$$      
The vector $(\bar{\a}_t^{1},...,\bar{\a}_t^{N})_{t\geq  0}$ defines a Nash equilibrium on $[0,T]$ for any initial data.
\item \label{Assert3}
The following holds
$$
\la^{i}\left[  \hat{\nu}_{ \bar{X}_{t}^{-i}}^{N-1} \right] = \liminf_{\tau \rightarrow +\infty} \frac{1}{\tau}\E \left[ \int_{t}^{\tau} L^{i}(\bar{\X}_{s,t}^{i},\bar{\a}_{t}^{i}(\bar{\X}_{s,t}^{i}))+F^{i}\left(\bar{\X}_{s,t}^{i} , \hat{\nu}_{ \bar{X}_{t}^{-i}}^{N-1}\right) \dd s \ \Big| \ \F_{t}  \right],$$
where $(\bar{\X}_{s,t}^{i})_{s>t}$ are obtained by solving 
\begin{equation}\nonumber
\left\{
\begin{aligned}
&\dd \bar{\X}_{s,t}^{i}= \bar{\a}_{t}^{i}(\bar{\X}_{s,t}^{i}) \dd s + \sqrt{2\s'_{i}}\dd \B_{s-t,t}^{i},\quad s>t,\\
&\bar{\X}_{t,t}^{i}=\bar{X}_{t}^{i}, \quad i=1,...,N.
\end{aligned}
\right.
\end{equation}
\end{enumerate}
\end{proposition}
\begin{proof}
Assertion \ref{Assert1} is a consequence of the regularity results of Lemma \ref{Lemma-interminable}, while
assertions \ref{Assert2} and \ref{Assert3} follows by standard verification arguments (see e.g. \cite[Theorem 3.4]{Feleqi}, among many others). For any $ \tau> t \geq 0$ and $i\in\{1,..,N\}$ one has
\begin{multline}
u^{i}(\bar{\X}_{\tau,t}^{i})= u^{i}(\bar{\X}_{t,t}^{i})+\int_{t}^{\tau} Du^{i}(\bar{\X}_{s,t}^{i}). \bar{\a}_{t}^{i}(\bar{\X}_{s,t}^{i}) \dd s + \int_{t}^{\tau} \s'_{i}\D u^{i}(\bar{\X}_{s,t}^{i})\dd s \\ \nonumber
+ \sqrt{2\s'_{i}}\int_{t}^{\tau}D u^{i}(\bar{\X}_{s,t}^{i})\dd \B_{s-t,t}^{i},
\end{multline}
where here $u^{i}\equiv u^{i}\left[ \hat{\nu}_{ \bar{X}_{t}^{-i}}^{N-1}\right]$ in order to simplify the presentation.
Owing to \eqref{optim-prop21-learning} one gets
\begin{multline}
u^{i}(\bar{\X}_{\tau,t}^{i}) =  u^{i}(\bar{\X}_{t,t}^{i})+\int_{t}^{\tau} \left( -H^{i}(\bar{\X}_{s,t}^{i},D u^{i}(\bar{\X}_{s,t}^{i}))+\s'_{i}\D u^{i}(\bar{\X}_{s,t}^{i})\right) \dd s  \\ \nonumber
+  \sqrt{2\s'_{i}}\int_{t}^{\tau}D u^{i}(\bar{\X}_{s,t}^{i})\dd \B_{s-t,t}^{i} -  \int_{t}^{\tau} L^{i}(\bar{\X}_{s,t}^{i}, \bar{\a}_{t}^{i}(\bar{\X}_{s,t}^{i})) \dd s \\ \nonumber
= u^{i}(\bar{\X}_{t,t}^{i}) - \int_{t}^{\tau} \left\{ L^{i}(\bar{\X}_{s,t}^{i}, \bar{\a}_{t}^{i}(\bar{\X}_{s,t}^{i})) + F^{i}(\bar{\X}_{s,t}^{i}; \bar{X}_{t}^{-i})  \right\} \dd s \\ \nonumber
 +(\tau-t)\la_{i} +\sqrt{2\s'_{i}}\int_{t}^{\tau}D u^{i}(\bar{\X}_{s,t}^{i})\dd \B_{s-t,t}^{i}.
\end{multline}
Hence, from \eqref{Hypo-indép}  we infer that
\begin{multline}
\tau^{-1}\E\left[u^{i}(\X_{\tau,t}^{i})\ | \ \F_{t}\right]= (1-t\tau^{-1})\la_{i}+ \tau^{-1}\E\left[u^{i}(\X_{t,t}^{i})\ | \ \F_{t}\right] \\ \nonumber
- \tau^{-1}\E\left[\int_{t}^{\tau} L^{i}(\bar{\X}_{s,t}^{i}, \bar{\a}_{t}^{i}(\bar{\X}_{s,t}^{i})) + F^{i}\left(\bar{\X}_{s,t}^{i} , \hat{\nu}_{ \bar{X}_{t}^{-i}}^{N-1}\right) \dd s \ \Big| \ \F_{t} \right].
\end{multline}
Note that estimate \eqref{ergodicproblem-interminable-2} provides a uniforme bound on $u^{i}\left[ .\right]$. Thus, by taking the limit in the last expression one gets
$$ 
\la_{i}\left[  \hat{\nu}_{ \bar{X}_{t}^{-i}}^{N-1} \right] = \liminf_{\tau \rightarrow +\infty} \frac{1}{\tau}\E \left[ \int_{t}^{\tau} L^{i}(\bar{\X}_{s,t}^{i},\bar{\a}_{t}^{i}(\bar{\X}_{s,t}^{i}))+F^{i}\left(\bar{\X}_{s,t}^{i} , \hat{\nu}_{ \bar{X}_{t}^{-i}}^{N-1}\right) \dd s \ \Big| \ \F_{t}  \right].
$$
On the other hand, one easily checks that  $(\bar{\a}_{t}^{1},...,\bar{\a}_{t}^{N})_{t\geq0}$ is a Nash equilibrium for any initial data $V=(V^{1},...,V^{N})$ owing to \eqref{optim-prop21-learning}.
\end{proof}

\begin{remark}
Note that the problem structure decouples the ``fictitious" dynamics \eqref{realdyn-eq1,2}, and allows to compute the controls.
\end{remark}

\subsubsection{The case of a discounted cost functional}  Set $\rho^{1},...,\rho^{N}>0$.
We consider now the case where the $i$-th player seeks to minimize the following discounted cost functional:
\begin{equation}\label{learning-cost-perfect-info-discounted}
J_{\rho_{i}}^{i}\left(t,V,\bar{\a}_{t}^{1},...,\bar{\a}_{t}^{N}\right) := \E \left[ \int_{t}^{\infty} e^{-\rho^{i}s} L^{i}(\X_{s,t}^{i},\a_{t}^{i}(\X_{s,t}^{i}))+F^{i}\left(\X_{s,t}^{i} ; \hat{\nu}_{ X_{t}^{-i}}^{N-1} \right) \dd s \  \Big| \ \mathcal{F}_{t}  \right],
\end{equation}
where all the functions are defined in the same way as in the previous case, with analogous notations and assumptions. 
One checks that  a similar result to Proposition \ref{proposition-learning} holds, i.e. that the following problem: 
\begin{equation}\label{Jeux-1-complete-info-bis-1} \dd \bar{Z}_{t}^{i}= \bar{\a}^{i}\left(\bar{Z}_{t}^{i},Dv_{\rho}^{i}\left[   \hat{\nu}_{ \bar{Z}_{t}^{-i}}^{N-1}\right]\left(\bar{Z}_{t}^{i}\right)\right) \dd t  +\sqrt{2\s_{i}}\dd W_{t}^{i},\quad \bar{Z}_{0}^{i}=V^{i}, \quad i=1,...,N,
\end{equation}
characterizes a Nash equilibrium on $[0,T]$ associated to the cost functional \eqref{learning-cost-perfect-info-discounted}.

\subsection{ The mean field limit $N\rightarrow +\infty$}\label{The mean field limit}  We address now the convergence problem when the number of players goes to infinity, assuming that all the players are indistinguishable. 

Assume that: 
$D^{i}=D$; $\rho_{i}=\rho$; $\s_{i}=\s$; $\s'_{i}=\s'$; $F^{i}=F$; $H^{i}=H$; and $\bar{\a}^{i}=\bar{\a}$ so that $L^{i}=H^{\ast}$ for all $1\leq i \leq N$, where $H^{\ast}$ is the \emph{Legendre transform} of $H$ with respect to the $p$ variable. We suppose also that
$$
\mathcal{L}(V^{i}) = m_{0}\in C^{2+\gamma}(Q) \ \ \mbox{for any}\ \ i=1,...,N.
$$
For simplicity 
we shall use the notations $X_{t}:=\left(X_{t}^{1},...,X_{t}^{N}\right)$ and  $Z_{t}:=\left(Z_{t}^{1},...,Z_{t}^{N}\right)$ instead of $\bar{X}_{t}:=\left(\bar{X}_{t}^{1},...,\bar{X}_{t}^{N}\right)$ and $\bar{Z}_{t}:=\left(\bar{Z}_{t}^{1},...,\bar{Z}_{t}^{N}\right)$. Under the above assumptions, systems  \eqref{Jeux-1-complete-info-1}  and \eqref{Jeux-1-complete-info-bis-1} are rewritten respectively on the following form:
\begin{equation}
\label{Learning-equation-form-bis-1}  
\left\{
\begin{aligned}
&\dd X_{t}^{i}= -H_{p}\left(X_{t}^{i}, Du\left[\hat{\nu}_{X_{t}^{-i}}^{N-1}\right](X_{t}^{i})\right) \dd t  +\sqrt{2\s}\dd W_{t}^{i}, \quad 0\leq t \leq T,\\
\\
& \bar{X}_{0}^{i}=V^{i}, \quad i=1,...,N;
\end{aligned}
\right.
\end{equation}
and 
\begin{equation}
\label{Learning-equation-form-bis-2}  
\left\{
\begin{aligned}
&\dd Z_{t}^{i}= -H_{p}\left(Z_{t}^{i}, Dv_{\rho}\left[\hat{\nu}_{Z_{t}^{-i}}^{N-1}\right](Z_{t}^{i})\right) \dd t  +\sqrt{2\s}\dd W_{t}^{i}, \quad 0\leq t \leq T,\\
\\
&\bar{Z}_{0}^{i}=V^{i}, \quad i=1,...,N.
\end{aligned}
\right.
\end{equation}

Our main result in this section  says that at the mean field limit $N\to\infty$, one recovers the quasi-stationary systems \eqref{QSS-intro-1} and \eqref{QSS-intro-2}, which respectively correspond  to \eqref{Learning-equation-form-bis-1} and \eqref{Learning-equation-form-bis-2}. Note that systems \eqref{QSS-intro-1} and \eqref{QSS-intro-2} can be rewritten on the form of Mckean Vlasov equations:
\begin{equation}
\tag{\ref{QSS-intro-1}}
\left\{
\begin{aligned}
&\partial_{t}m-\s\D m-\div\left(m H_{p}\left(x,Du[m(t)]\right)\right)=0\quad \mbox{ in } Q_{T}, \\     
\\
& m(0)=m_{0}\quad \mbox{ in } Q, 
\end{aligned}
\right.
\end{equation}
and
\begin{equation}
\tag{\ref{QSS-intro-2}}
\left\{
\begin{aligned}
&\partial_{t}\mu-\s\D \mu-\div\left(\mu H_{p}\left(x,Dv_{\rho}[\mu(t)]\right)\right)=0\quad \mbox{ in } Q_{T}, \\     
\\
& \mu(0)=m_{0}\quad \mbox{ in } Q.
\end{aligned}
\right.
\end{equation}
Thus, one can use the usual coupling arguments (see e.g. \cite{Snitzman, Mckean, Méléard}) to deduce the convergence. The main theorem of this section is the following:
\begin{theorem}\label{theorem-convergence-N}
For any $t\in[0,T]$, it holds that:
\begin{equation}
\nonumber
\begin{split}
&\lim_{N} \max_{1\leq i\leq N}\ddd_{1}\left(\mathcal{L}\left(X_{t}^{i}\right),m(t)\right) =0; \\
&\lim_{N} \max_{1\leq i\leq N}\ddd_{1}\left(\mathcal{L}\left(Z_{t}^{i}\right),\mu(t)\right) =0;\\
&\lim_{N}\left\| u[m(t)] - \E u\left[\hat{\nu}_{X_{t}}^{N}\right]\right\|_{\infty} =0; \\
&\lim_{N}\left| \la[m(t)] - \E\la\left[\hat{\nu}_{X_{t}}^{N}\right] \right|=0; \quad \mbox{and}\\
&\lim_{N}\left\| v_{\rho}[\mu(t)] - \E v_{\rho}\left[\hat{\nu}_{Z_{t}}^{N}\right] \right\|_{\infty} =0.
\end{split}
\end{equation}
\end{theorem}

The analysis of the limit transition $N\to +\infty$ is essentially based on continuous dependence estimates, and therefore the mean field analysis is identical for both systems. Thus, we shall give the details only for system \eqref{Learning-equation-form-bis-1}. 

Let us introduce the following artificial systems:
\begin{equation}
\label{artificial-system-1}  
\left\{
\begin{aligned}
&\dd Y^{i}= - H_{p}\left(Y_{t}^{i}, Du\left[m(t)\right](Y_{t}^{i})\right)\dd t+\sqrt{2\s}\dd W_{t}^{i}, \quad 0\leq t \leq T,\\
\\
& \bar{Y}_{0}^{i}=V^{i}, \quad i=1,...,N;
\end{aligned}
\right.
\end{equation}
and
\begin{equation}\label{intermediate-system-1}
\left\{
\begin{aligned}
&\dd \tilde{X}^{i}= - H_{p}\left(\tilde{X}_{t}^{i}, Du\left[ \hat{\nu}_{X_{t}}^{N} \right](\tilde{X}_{t}^{i})\right) \dd t+\sqrt{2\s}\dd W_{t}^{i}, \quad 0\leq t \leq T, \\
\\
&\tilde{X}_{0}^{i} = V^{i},  \quad i=1,...,N.
\end{aligned}
\right.
\end{equation}
Observe that systems \eqref{artificial-system-1}-\eqref{intermediate-system-1} are well-posed, and that the uniqueness of the solution to \eqref{QSS-intro-1} provides that
$$
\mathcal{L}\left(Y_{t}^{1},...,Y_{t}^{N}\right) = \otimes_{i=1}^{N} m(t).
$$
On the other hand, note that
\begin{equation}\label{symmetry-joint-proba}
\mathcal{L}\left( \tilde{X}^{\xi(1)}_{t},..., \tilde{X}^{\xi(N)}_{t}\right) =\mathcal{L}\left( \tilde{X}^{1}_{t},..., \tilde{X}^{N}_{t}\right)
\end{equation}
is fulfilled for any permutation $\xi$, and any $t\in[0,T]$. In addition, one checks that
\begin{equation}\label{first-estimation-first-part}
\max_{1\leq i\leq N}\sup_{0\leq t \leq T}\E\left | X_{t}^{i} -\tilde{X}_{t}^{i} \right| \leq \frac{C_T}{N-1}
\end{equation}
holds thanks to the continuous dependence estimate \eqref{ergodicproblem-interminable-2}, since
$$
\sup_{0\leq t\leq T}\max_{1\leq i\leq N}\ddd_{1}\left(\hat{\nu}_{X_{t}}^{N}, \hat{\nu}_{X_{t}^{-i}}^{N-1}\right) \leq \frac{C}{N-1}.
$$

Next we compare the trajectories of \eqref{artificial-system-1} and  \eqref{intermediate-system-1}, and show that they are increasingly close on $[0,T]$ when $ N \to + \infty$.

\begin{proposition}\label{key-proposition-chao} Under assumptions of this section, it holds that
$$
\max_{1\leq i\leq N}\sup_{0\leq t \leq T}\E\left | \tilde{X}_{t}^{i} -Y_{t}^{i} \right| \leq C_{T}N^{-1/(d+8)}.
$$
\end{proposition}
\begin{proof}
For any $i\in\{ 1,...,N\}$ and $t\in[0,T]$, one has 
\begin{eqnarray}
\frac{\dd}{\dd t} \left[ \tilde{X}_{t}^{i} - Y_{t}^{i}\right] &=& H_{p}\left(Y_{t}^{i}, Du[m(t)](Y_{t}^{i})\right)-H_{p}\left(\tilde{X}_{t}^{i}, Du\left[\hat{\nu}_{X_{t}}^{N}\right](\tilde{X}_{t}^{i})\right) \nonumber \\
&=& \mathcal{T}_{1}(t)+\mathcal{T}_{2}(t)+\mathcal{T}_{3}(t)+\mathcal{T}_{4}(t), \nonumber
\end{eqnarray} 
where 
$$
\mathcal{T}_{1}(t):= H_{p}(Y_{t}^{i}, Du[m(t)](Y_{t}^{i}))-H_{p}(Y_{t}^{i}, Du\left[\hat{\nu}_{Y_{t}}^{N}\right](Y_{t}^{i})),
$$
$$
\mathcal{T}_{2}(t) := H_{p}(Y_{t}^{i}, Du\left[\hat{\nu}_{Y_{t}}^{N}\right](Y_{t}^{i}))-H_{p}(Y_{t}^{i}, Du\left[\hat{\nu}_{X_{t}}^{N}\right](Y_{t}^{i})),
$$
$$
\mathcal{T}_{3}(t) := H_{p}(Y_{t}^{i}, Du\left[\hat{\nu}_{X_{t}}^{N}\right](Y_{t}^{i})) - H_{p}(\tilde{X}_{t}^{i}, Du\left[\hat{\nu}_{X_{t}}^{N}\right](Y_{t}^{i})),
$$
and
$$
\mathcal{T}_{4}(t) := H_{p}(\tilde{X}_{t}^{i}, Du\left[\hat{\nu}_{X_{t}}^{N}\right](Y_{t}^{i})) - H_{p}(\tilde{X}_{t}^{i}, Du\left[\hat{\nu}_{X_{t}}^{N}\right](\tilde{X}_{t}^{i})).
$$
Using the continuous dependence estimate \eqref{ergodicproblem-interminable-2} one gets
$$
|\mathcal{T}_{2}(t) | \leq \frac{C}{N} \sum_{j=1}^{N} |\tilde{X}_{t}^{j} - Y_{t}^{j}|, \quad \mbox{ and } \quad  |\mathcal{T}_{1}(t) | \leq C\ddd_{1}(m(t),\hat{\mu}_{Y_{t}}^{N}).
$$
On the other hand, the following holds
$$
\left| \mathcal{T}_{3}(t)+\mathcal{T}_{4}(t) \right| \leq C | \tilde{X}_{t}^{i} - Y_{t}^{i} |.
$$

The key step is the estimation the non-local term $\E\ddd_{1}(m,\hat{\mu}_{Y}^{N})$; we use the following estimate due to Horowitz and Karandikar (see \cite[Theorem 10.2.7]{product:measures}):
$$
\E\ddd_{1}(m(t),\hat{\mu}_{Y_{t}}^{N}) \leq \k_{d} N^{-1/(d+8)} \quad \forall t\in[0,T],
$$
where the constant $\k_{d}>0$ depends only on $d$.
Using the symmetry of the joint probability law \eqref{symmetry-joint-proba} and the last estimate, we infer that
$$
 \E \left| \tilde{X}_{t}^{i} - Y_{t}^{i}  \right| \leq C\int_{0}^{t}\left(\frac{1}{N^{1/(d+8)}} + \E \left| \tilde{X}_{s}^{i} - Y_{s}^{i}  \right|\right)\dd s,
$$
which concludes the proof.
\end{proof}

Recall the following definition and characterizations of chaotic measures \cite{Snitzman}. 
\begin{definition}
Let $\pi^{N}$ be a symmetric joint probability measure on $Q^{N}$ and $\pi \in \P(Q)$. We say that $\pi^{N}$ is $\pi$-chaotic if for any $k\geq 1$ and any continuous functions $\phi_{1},...,\phi_{k}$ on $Q$ one has
$$
\lim_{N} \int \prod_{l=1}^{k}\phi_{l}\dd \pi^{N} =  \prod_{l=1}^{k} \int  \phi_{l}\dd \pi.
$$
\end{definition}

\begin{lemma}\label{lemma-chaos-propagation}
Let $\mathbb{X}_{N}$ be a sequence of random variables on $Q^{N}$  whose the joint probability law $\pi^{N}$ is symmetric, and $\pi \in \P(Q)$. Then  the following assertions are equivalent:
\begin{enumerate}[ label=\upshape(\roman*)]
\item $\pi^{N}$ is $\pi$-chaotic;
\item\label{assertion2} the empirical  measure $\hat{\nu}_{\mathbb{X}_{N}}^{N}$ converges in law toward the deterministic measure $\pi$;
\item for any continuous function $\phi$ on $Q$, it holds that
$$
\lim_{N}\E \left| \int \phi \dd \left(\hat{\nu}_{\mathbb{X}_{N}}^{N} - \pi\right)  \right| =0.
$$
\end{enumerate}
\end{lemma}

Combining Lemma \ref{lemma-chaos-propagation} and Proposition \ref{key-proposition-chao}, we deduce
\emph{the propagation of chaos} for system  \eqref{intermediate-system-1}.
\begin{proposition}\label{chaotic-limit} For any $t\in[0,T]$, if
$m^{N}(t)$ is the joint probability law of $\tilde{X}_{t}:=(\tilde{X}^{j}_{t})_{1\leq j \leq N}$, then $m^{N}(t)$ is $m(t)$-chaotic.
\end{proposition}
\begin{proof}
Let $\phi$ be a Lipschitz continuous function on $Q$. 
From Proposition \ref{key-proposition-chao}, we have that
$$
\E \left| \int \phi \dd (\hat{\mu}_{\tilde{X}_{t}}^{N} - \hat{\mu}_{Y_{t}}^{N})  \right| \leq \frac{\| \phi \|_{\Lip} }{N}  \sum_{k=1}^{N} \E \left| \tilde{X}_{t}^{i}-Y_{t}^{i} \right| \leq \frac{\| \phi \|_{\Lip} C_{T}}{N^{1/(d+8)}}.
$$
Invoking the fact that $\mathcal{L}(Y_{t}) = \otimes_{i=1}^{N}m(t)$ and Lemma \ref{lemma-chaos-propagation}, it holds that
$$
\lim_{N}\E \left| \int \phi \dd (m(t) - \hat{\mu}_{Y_{t}}^{N})  \right| =0.
$$ 
The claimed result follows from
$$
\E \left| \int \phi \dd (\hat{\mu}_{\tilde{X}_{t}}^{N} - m(t))  \right| \leq \E \left| \int \phi \dd (\hat{\mu}_{\tilde{X}_{t}}^{N} - \hat{\mu}_{Y_{t}}^{N})  \right| +  \E \left| \int \phi \dd (m(t) - \hat{\mu}_{Y_{t}}^{N})  \right|.
$$
\end{proof}

We are now in position to prove Theorem \ref{theorem-convergence-N}.
\begin{proof}[Proof of Theorem \ref{theorem-convergence-N}]
Observe that for any two random variables $\mathbb{X}$, $\mathbb{Y}$, one has
$$
\ddd_{1}\left( \mathcal{L}(\mathbb{X}), \mathcal{L}(\mathbb{Y}) \right) \leq \E\left| \mathbb{X}-\mathbb{Y} \right|.
$$
Hence, combining Lemma \ref{key-proposition-chao} and estimate \eqref{first-estimation-first-part} one gets that
$$
\lim_{N} \max_{1\leq i\leq N}\ddd_{1}\left(\mathcal{L}\left(X_{t}^{i}\right),m(t)\right) =0.
$$
On the other hand, we have
$$
\lim_{N}\left| \la[m(t)] - \E\la\left[\hat{\nu}_{\tilde{X}_{t}}^{N}\right] \right|=0 \quad \mbox{and} \quad
 \lim_{N}\left\| u[m(t)] - \E u\left[\hat{\nu}_{\tilde{X}_{t}}^{N}\right]\right\|_{\infty} =0,
$$
thanks to Proposition \ref{chaotic-limit} and Lemma \ref{lemma-chaos-propagation}. In fact, pointwise convergence 
is a consequence of assertion \ref{assertion2} in Lemma \ref{lemma-chaos-propagation}, and the convergence is actually uniform  since  $u\left[ \P(Q) \right]$ is compact in $C(Q)$. We conclude the proof for $(\la,u,m)$ by invoking \eqref{first-estimation-first-part} and the continuous dependence estimates \eqref{ergodicproblem-interminable-1}-\eqref{ergodicproblem-interminable-2}.
The results for $(v,\mu)$ follows using similar steps as for $(\la, u, m)$. 
\end{proof}
\begin{remark}
Note that the two main arguments in the proof of Theorem \ref{theorem-convergence-N} are the continuous dependence estimate, and symmetry with respect to states of the other players. 
\end{remark}

\section{ Exponential convergence to the ergodic MFG equilibrium }\label{section-exp-convergence}

We prove in this section the exponential convergence of the quasi-stationnary system \eqref{QSS-intro-1} to the ergodic equilibrium assuming that $\s'=\s$ and $H(x,p)=|p|^{2}/2$. 
The proofs rely on algebraic properties of the equations, the continuous dependence estimates (Lemma \ref{ergodicproblem-interminable}), and the monotonicity condition \eqref{monotonicity-condition}. Throughout this section we suppose that
assumptions \ref{A1},\ref{A2}, and \ref{A6} are fulfilled. 
In addition, we assume that the coupling $F$ satisfies the monotonicity condition:
\begin{equation}\tag{\ref{monotonicity-condition}}
\forall m,m'\in\mathcal{P}(Q),\qquad \int_{Q}\left( F(x,m)-F(x,m')  \right)\dd(m-m')(x) \geq 0.
\end{equation}
For the sake of simplicity we set $\s=\s'=1$.

In this framework the quasi-stationary MFG system \eqref{QSS-intro-1} takes the following form,
\begin{equation}
\label{QSSquad}
\left\{
\begin{aligned}
&\la(t)-\D u+\frac{1}{2}\vert Du\vert^{2}=F(x,m(t))\quad\mbox{ in } (0,\infty)\times Q,\\
\\
&\partial_{t}m-\D m-\div(mDu)=0\quad\mbox{ in } (0,\infty)\times Q, \\      
\\
& m(0)=m_{0} \quad\mbox{ in } Q, \quad <u>=0 \quad\mbox{ in } (0,\infty).
\end{aligned}
\right.
\end{equation}
System \eqref{QSSquad} has a unique  global (in time) classical solution thanks to Theorem \ref{existence-QSS-a}. 
Consider the following ergodic Mean Field Games problem:
\begin{equation}
\label{EMFG} 
\left\{
\begin{aligned}
&\bar{\la}-\D \bar{u}+\frac{1}{2}\vert D \bar{u} \vert^{2}=F(x,\bar{m}) \quad\mbox{ in }  Q,\\
\\
&-\D \bar{m}-\div(\bar{m}D\bar{u})=0 \quad\mbox{ in } Q, \\      
\\
&\bar{m}\geq 0 \quad\mbox{ in } Q,\quad <\bar{m}>=1,\quad <\bar{u}>=0  . \qquad 
\end{aligned}
\right.
\end{equation} 
Under the  monotonicity condition \eqref{monotonicity-condition}, uniqueness holds for system \eqref{EMFG}.
In all this section $(\bar{\la},\bar{u}, \bar{m})$ denotes the unique solution to \eqref{EMFG}.
Observe that 
$
\bar{m}\equiv e^{-\bar{u}}/ \left<e^{-\bar{u}}\right>,
$
so that the following holds
\begin{equation}\label{explicit-formulas}
1/\bar{\k}\leq \bar{m}\leq \bar{\k}
\end{equation}
for some constant  $\bar{\k}>0$.

The main result of this section is the following: 

\begin{theorem}\label{maintheorem2}
There exists $R_{0}>0$ such that if 
$$
\| m_{0}-\bar{m}\|_{2}\leq R_{0},
$$
then the following holds for some constants $A,\d >0$:
$$
\vert \la(t)-\bar{\la}\vert+\Vert u(t)-\bar{u}\Vert_{C^{2}}+\Vert m(t)-\bar{m}\Vert_{2} \leq A e^{-\d t}\quad \mbox{ for any } t\geq 0.
$$
\end{theorem}
This convergence result reveals that our decision-making mechanism lead to the emergence of a Mean Field Games equilibrium, under the conditions mentioned above. This can also be interpreted  as a phase transition from  a non-equilibrium state to an equilibrium state (see also \cite{Mehta1, Mehta2}).
Agents reach this equilibrium by adjusting and self-correcting.
We believe that this convergence result holds true in more general cases. For instance, one can show that an analogous convergence result holds for system \eqref{QSS-intro-2} when the discount rate $\rho$ is small enough (c.f. Remark \ref{remark-extension}).

Let $(\la,u,m)$ be the solution to \eqref{QSSquad}, and set
\begin{equation}\label{notation requise-411}
\varsigma := \lambda-\bar{\la},\quad w:=u-\bar{u}, \quad \mbox{ and  }\pi:=m-\bar{m}.
\end{equation}
The triplet $(\varsigma, w, \pi)$ is a solution to the following system of equations:
\begin{equation}
\label{QSSNL} 
\left\{
\begin{aligned}
&\varsigma(t)-\D w+\langle D\bar{u}, D w \rangle+\frac{1}{2}\vert Dw\vert^{2}=F(x,\bar{m}+\pi(t))-F(x,\bar{m})  \quad\mbox{ in } (0,\infty)\times Q,\\
\\
&\partial_{t}\pi-\D \pi-\div(\pi D \bar{u})-\div(\bar{m}D w)-\div(\pi D w)=0 \quad\mbox{ in } (0,\infty)\times Q, \\      
\\
& \pi(0)=m_{0}-\bar{m} \quad\mbox{ in } Q, \quad <w>=0 \quad\mbox{ in } (0,\infty).
\end{aligned}
\right.
\end{equation}

The following preliminary Lemma states the dependence of $w$  and $\varsigma$ on $\pi$ in the first equation of \eqref{QSSNL}.

\begin{lemma}\label{lemma-jdid1}
Let $\varpi$ be a probability measure on $Q$ which is absolutely continuous with respect to the Lebesgue measure, and such that  
$$
\varpi=\bar{m}+\pi,
$$
where  $\pi \in L^{2}(Q)$. There exists a unique periodic solution $(\varsigma[\pi], w[\pi])$ in $\R\times C^{2}(Q)$ to the following problem:
\begin{equation}
\label{QSS-aux-prob} 
\left\{
\begin{aligned}
&\varsigma-\D w+\langle D\bar{u}, D w \rangle+\frac{1}{2}\vert Dw\vert^{2}=F(x,\varpi)-F(x,\bar{m}) \quad\mbox{ in } (0,\infty)\times Q\\
&<w>=0 \quad\mbox{ in } (0,\infty).
\end{aligned}
\right.
\end{equation}
Moreover, the following estimates hold
\begin{subequations}
\begin{equation}\label{estim11-prfTH23}
\left\vert \varsigma[\pi]\right\vert \leq C \| \pi\|_{2},
\end{equation}
\begin{equation}\label{estim1-prfTH23}
\| w[\pi]\|_{C^{2}} \leq C' \| \pi\|_{2}.
\end{equation}
\end{subequations}
\end{lemma}
\begin{proof}
Existence and uniqueness of regular solutions to such problems are discussed in Section \ref{section2}.
Estimates \eqref{estim11-prfTH23}-\eqref{estim1-prfTH23} are a direct consequence of the uniqueness and the continuous dependence estimates \eqref{ergodicproblem-interminable-1}-\eqref{ergodicproblem-interminable-2}.
\end{proof}

Next we give the following technical Lemma.

\begin{lemma}\label{technical-lemma-poincarre}
There exists a constant $\k>0$ such that
$$
\| \pi/\bar{m}\|_{2} \leq \k \| D(\pi/\bar{m}) \|_{2},
$$
for any $\pi \in V^{1,2}(Q):=\{\pi\in W^{1,2}(Q) \ / \ <\pi> =0 \}$.
\end{lemma}
\begin{proof}
As usual, the result is obtained by contradiction. In fact, if our claim is not satisfied one can find a sequence $(\pi_{n})\in V^{1,2}(Q)$ such that  for any $n\geq 1$,
\begin{equation}\label{eq-lemme21-tech1}
\| \pi_{n}/\bar{m} \|_{2}=1 \quad \mbox{ and }\quad  \frac{1}{n} \geq \| D(\pi_{n}/\bar{m}) \|_{2}.
\end{equation}
By Sobolev embeddings,  $(\pi_{n}/\bar{m})_{n}$ converges (up to a subsequence) to some $\bar{\pi}$ in $L^{2}(Q)$. Using \eqref{eq-lemme21-tech1} it follows that $\bar{\pi}$ is constant, i.e. $\bar{\pi}\equiv C$. Moreover,
$$
 C=\int_{Q}\bar{m}\bar{\pi}=\lim_{n}\int_{Q}\pi_{n}=0;
$$
this provides the desired contradiction owing to \eqref{eq-lemme21-tech1}.
\end{proof}
Combining these elements one can prove the main Theorem of this section.

\begin{proof}[Proof of Theorem~\ref{maintheorem2}]
Let $(\varsigma,w, \pi)$ be a smooth solution to \eqref{QSSNL}. 
Recall that
$$D \bar{m} = -e^{-\bar{u}}D \bar{u}/<e^{-\bar{u}}> = -\bar{m}D\bar{u},$$
so that
\begin{equation}\label{estim2-prflemma2.2}
D\left(\frac{\pi}{\bar{m}} \right)=\frac{D\pi+\pi D\bar{u}}{\bar{m}},
\end{equation}
and
\begin{equation}\label{estim1-prflemma2.2}
\div(\bar{m}Dw)=\bar{m}\D w+\langle D\bar{m},Dw \rangle=-\bar{m}\left( -\D w+\langle D\bar{u} , D w \rangle  \right).
\end{equation}
We infer that
$$
\div(\bar{m}D w)= -\bar{m} \left( -\varsigma(t)-\frac{1}{2}\vert Dw \vert^{2}+F(x,\bar{m}+\pi(t))-F(x,\bar{m})  \right), 
$$
which provides in particular
$$
\p_{t} \pi-\D \pi -\div(\pi D \bar{u})+\bar{m}\left( -\varsigma(t)-\frac{1}{2}\vert D w\vert^{2}+F(x,\bar{m}+\pi(t))-F(x,\bar{m}) \right)-\div(\pi D w)=0.
$$
Hence, using \eqref{estim2-prflemma2.2}, \eqref{estim1-prfTH23} and the monotonicity of $F$, one has
\begin{eqnarray} \nonumber
\frac{\dd }{\dd t} \int_{Q} \frac{\pi^{2}}{2\bar{m}} &=& \int_{Q} \frac{\pi}{\bar{m}}\Big[ \D \pi+\div(\pi D \bar{u})+\div(\pi D w) \\ \nonumber
&&\qquad -\bar{m}\left( -\varsigma(t)-\frac{1}{2}\vert D w \vert^{2}+F(x,\bar{m}+\pi(t))-F(x,\bar{m})  \right)  \Big]\nonumber \\ \nonumber
&=& \int_{Q} \left\{-\frac{\vert D\pi+\pi D\bar{u} \vert^{2}}{\bar{m}} + \frac{\pi \vert Dw\vert^{2}}{2} - \frac{\pi \langle  D\pi, D w\rangle}{\bar{m}} - \frac{\pi^2\langle D\bar{u}, D w \rangle}{\bar{m}} \right\} \\ \nonumber &&\qquad- \int_{Q} \pi\left( F(x,\bar{m}+\pi(t))-F(x,\bar{m}) \right)\\ \nonumber 
&\leq& \int_{Q} -\bar{m} \left\vert D\left(\frac{\pi(t)}{\bar{m}}  \right) \right\vert^{2} + C\|\pi(t)\|_{2}^{3}+1/\bar{\k}\| D\pi(t)\|_{2}\| \pi(t)\|_{2}^{2}.\\ \nonumber
&\leq& -1/\bar{\k}\int_{Q} \left\vert D\left(\frac{\pi(t)}{\bar{m}}  \right) \right\vert^{2} + C\left(\|\pi(t)\|_{2}+\| D\pi(t)\|_{2}\right)\| \pi(t)\|_{2}^{2}.  \\
\label{estim1bis-prfth2.3}
\end{eqnarray}
Using  Lemma \ref{technical-lemma-poincarre}, one easily checks that 
\begin{equation}\label{estim-importante}
\| \pi(t) \|_{2}+\|D\pi(t)\|_{2} \leq C \left\| D\left(\frac{\pi(t)}{\bar{m}}  \right) \right\|_{2}.
\end{equation}
Thus the following holds 
\begin{equation}\label{estimation-clé-prfth2.4}
\frac{\dd}{\dd t}\| \pi(t)^2/\bar{m} \|_{1} \leq -1/C_{0}\| \pi(t)^{2}/\bar{m} \|_{1}+M\| \pi(t)^2/\bar{m} \|_{1}^{2},
\end{equation}
for some $C_{0},M>0$, thanks to \eqref{explicit-formulas} and Young's inequality. For any $R_{0}<\frac{1}{\sqrt{\bar{\k}MC_{0}}}$, the last differential inequality entails that 
$$
\| \pi(t)^2/\bar{m}\|_{1} \leq \frac{1/MC_{0}}{1+\left(\frac{1}{MC_{0}\| \pi(0)^2/\bar{m} \|_{1}}-1 \right)e^{t/C_{0}}} \quad \mbox{ for any } t\geq 0.
$$
Estimates of Lemma \ref{lemma-jdid1} conclude the proof.
\end{proof}

\begin{remark}\label{remark-extension}
One notices that the previous proof can be adapted to show that \eqref{QSS-intro-2} converges exponentially fast to \eqref{EMFG-intro-discount}  when the discount rate $\rho$ is small enough, under the same assumptions of Theorem \ref{maintheorem2}.
In fact, setting $\tilde \pi := \mu-\bar \mu$, the same arguments leading to \eqref{estimation-clé-prfth2.4} also provide
$$
\frac{\dd}{\dd t}\| \tilde \pi(t)^2/\bar{\mu} \|_{1} \leq -1/\tilde C_{0}\| \tilde \pi(t)^{2}/\bar{\mu} \|_{1}+\tilde M\| \tilde \pi(t)^2/\bar{\mu} \|_{1}^{2} + C\rho \| \tilde \pi(t)^{2}/\bar{\mu} \|_{1}.
$$
Therefore, the same conclusion holds when $\rho$ is small enough.
\end{remark}

\begin{remark}
In practice, this convergence results can help to understand the emergence of highly-rational equilibria in situations with myopic decision-making mechanisms. For instance in \cite[Section 2.2.2]{Cristiani} the authors consider a decision-making mechanism for pedestrian dynamics that is very similar to the mechanism described in Section \ref{section1bis}.
Theorem \ref{maintheorem2} can apply for this kind of models, in the case of a quadratic running cost, and a monotonous coupling function (pedestrians dislike congested areas) which satisfies \ref{A1} and \ref{A2}. For instance $F$ can take the following form:
$$
F(x,m)=\int_{\R^d} \phi \ast m(y) \phi(x-y) \dd y,
$$
where $\ast$ is the usual convolution product (in $\R^d$), and $\phi$ is a smooth, even function with compact support.
\end{remark}

\appendix
 
 \section{Elementary facts on the Fokker-Planck equation} 
 Let $V:[0,T]\times Q\rightarrow \R$ be a given bounded vector field, which is continuous in time and H{\"o}lder continuous in space, and we consider the following Fokker-Planck equation:
\begin{equation}
\label{elementary-FP}
\left\{
\begin{aligned}
&\p_{t}m-\sigma \D m-\div(m V)=0 \quad\mbox{ in } (0,T)\times Q,\\
& m(0)=m_{0} \quad\mbox{ in }  Q;
\end{aligned}
\right.
\end{equation}
and the following stochastic differential equation:
\begin{equation}
\label{elementary-SE}
\dd \mathbb{X}_{t}= V(t,\mathbb{X}_{t})\dd t+ \sqrt{2\sigma}\dd B_{t}\quad t\in(0,T], \quad \mathbb{X}_{0}= Z_{0},
\end{equation}
where $(B_{t})$ is a standard $d$-dimensional Brownian motion over some probability space $(\Omega, \mathcal{F},\mathbb{P})$ and $Z_{0}\in L^{1}(\Omega)$ is random and independent of $(B_{t})$. Under these  assumptions, there is a unique solution to \eqref{elementary-SE} and the following hold:
\begin{lemma}\label{FKP-Lemma}
If $m_{0}=\mathcal{L}(Z_{0})$ then, $m(t)=\mathcal{L}(\mathbb{X}_{t})$ is a weak solution to \eqref{elementary-FP} and there exists a constant $C_{T}>0$ such that, for any $t,s \in [0,T],$
$$
\ddd_{1}(m(t),m(s)) \leq C_{T}(1+\| V\|_{\infty})|t-s |^{1/2}.
$$
\end{lemma}
\begin{proof}
The first assertion is a straightforward consequence of It{\^o}'s formula. On the other hand, for any $1$-Lipschitz continuous function $\phi$ and any $t\geq s$, one has
\begin{eqnarray} \nonumber
\int_{\T^{d}}\phi(x) \dd(m(t)-m(s))(x) &\leq& \E\vert \phi(\mathbb{X}_{t})-\phi(\mathbb{X}_{s})   \vert  \leq \E\vert \mathbb{X}_{t}-\mathbb{X}_{s}   \vert \\ \nonumber
&\leq& \E \left[ \int_{s}^{t}\vert V(u,\mathbb{X}_{u}) \vert\dd u + \sqrt{2\s} \vert B_{t}-B_{s} \vert  \right] \\ \nonumber
&\leq& \| V \|_{\infty}(t-s)+ \sqrt{2\s(t-s)}.
\end{eqnarray}
\end{proof}

\bibliographystyle{siam}
\bibliography{C:/mybib/mybib}

\end{document}